\documentclass[twoside,leqno,symbols-for-thanks]{rmi}
\usepackage[colorlinks,linkcolor=black,citecolor=black,urlcolor=black, linktocpage=true]{hyperref}
\usepackage{srcltx}
\usepackage[english]{babel}

\numberwithin{equation}{section}
\setinitialpage{1}
\title[Embedding Heisenberg into Euclidean space]{Embedding the Heisenberg group into a bounded dimensional Euclidean space with optimal distortion}

\author[T. Tao]{Terence Tao}

\address[tao@math.ucla.edu]{UCLA Department of Mathematics, Los Angeles, CA 90095-1555.}

\amsclassification[]{30L05}

\keywords{Heisenberg group, distortion, Nash-Moser iteration}

\usepackage{amsthm}
\usepackage{amssymb}
\usepackage{mathtools}

\theoremstyle{plain}

\newtheorem{theorem}{Theorem}[section]
\newtheorem{proposition}[theorem]{Proposition}
\newtheorem{lemma}[theorem]{Lemma}
\newtheorem{corollary}[theorem]{Corollary}

\newtheorem{claim}[theorem]{Claim}

\theoremstyle{definition}

\newtheorem{remark}[theorem]{Remark}

\newcommand\C{\mathbb{C}}

\newcommand\R{\mathbb{R}}
\newcommand\Z{\mathbb{Z}}

\newcommand\eps{\varepsilon}

\begin{document}

%
%

\begin{abstract}
Let $H \coloneqq \begin{pmatrix} 1 & \R & \R \\ 0 & 1 & \R \\ 0 & 0 & 1 \end{pmatrix}$ denote the Heisenberg group with the usual Carnot-Carath\'eodory metric $d$.  It is known (since the work of Pansu and Semmes) that the metric space $(H,d)$ cannot be embedded in a bilipschitz fashion into a Hilbert space; however, from a general theorem of Assouad, for any $0 < \eps \leq 1/2$, the snowflaked metric space $(H,d^{1-\eps})$ embeds into an infinite-dimensional Hilbert space with distortion $O( \eps^{-1/2} )$. This distortion bound was shown by Austin, Naor, and Tessera to be sharp for the Heisenberg group $H$.  Assouad's argument allows $\ell^2$ to be replaced by $\R^{D(\eps)}$ for some dimension $D(\eps)$ dependent on $\eps$. Naor and Neiman showed that $D$ could be taken independent of $\eps$, at the cost of worsening the bound on the distortion to $O( \eps^{-1-c_D} )$, where $c_D \to 0$ as $D \to \infty$.  In this paper we show that one can in fact retain the optimal distortion bound $O( \eps^{-1/2} )$ and still embed into a bounded dimensional space $\R^D$, answering a question of Naor and Neiman.  As a corollary, the discrete ball of radius $R \geq 2$ in $\Gamma \coloneqq \begin{pmatrix} 1 & \Z & \Z \\ 0 & 1 & \Z \\ 0 & 0 & 1 \end{pmatrix}$ can be embedded into a bounded-dimensional space $\R^D$ with the optimal distortion bound of $O(\log^{1/2} R)$.

The construction is iterative, and is inspired by the Nash-Moser iteration scheme as used in the isometric embedding problem; this scheme is needed in order to counteract a certain ``loss of derivatives'' problem in the iteration.
\end{abstract}

%
%

\section{Introduction}

A map $f \colon (X,d_X) \to (Y,d_Y)$ between two metric spaces $(X,d_X), (Y,d_Y)$ is said to have \emph{distortion} at most $K$ if there exists a constant $C>0$ for which one has the bilipschitz bounds
$$ C d_X(x,x') \leq d_Y( f(x), f(x') ) \leq CK d_X(x,x').$$
There is considerable literature on the question of the optimal distortion in which one can embed a given metric space into a model metric space, such as an infinite-dimensional Hilbert space $\ell^2$, or a finite-dimensional Euclidean space $\R^D$ with the usual metric.  See for instance \cite{semmes}, \cite{lang}, \cite{heinonen} for some discussion.  In this paper we will restrict attention to the problem of embedding the Heisenberg group 
$$ H \coloneqq \begin{pmatrix} 1 & \R & \R \\ 0 & 1 & \R \\ 0 & 0 & 1 \end{pmatrix}$$
(or subsets thereof) into Euclidean spaces.  We make the abbreviation
$$ [x,y,z] \coloneqq \begin{pmatrix} 1 & x & z \\ 0 & 1 & y \\ 0 & 0 & 1 \end{pmatrix}$$
for $x,y,z \in \R$, thus
$$ H = \{ [x,y,z]: x,y,z \in \R \}$$
can be identified (as a set) with $\R^3$, and the group law is given by
$$ [a,b,c] [x,y,z] \coloneqq [x+a, y+b, z+c+ay].$$
This gives the right-invariant vector fields
$$ X \coloneqq \frac{\partial}{\partial x} + y \frac{\partial}{\partial z}; \quad Y \coloneqq \frac{\partial}{\partial y}; \quad Z \coloneqq \frac{\partial}{\partial z}$$
which we can apply to smooth\footnote{We endow $H$ with the usual smooth structure arising from its identification with $\R^3$.} vector-valued functions $\phi \colon H \to \R^D$ in the obvious fashion. We observe the Heisenberg relations
\begin{equation}\label{heisenberg}
Z = YX - XY; \quad XZ = ZX; \quad YZ = ZY.
\end{equation}
We let $d$ be the Carnot-Caratheodory metric on $H$ generated by $X,Y$; thus for $p,p' \in H$, $d(p,p')$ is the infimum of the quantity $|t_1|+\dots+|t_k|$ over all representations
$$ p' = \exp(t_1 W_1) \dots \exp(t_k W_k) p$$
where $t_1,\dots,t_k \in \R$, $W_1,\dots,W_k \in \{X,Y\}$, and $\exp(tW) \colon H \to H$ denotes the flow along the vector field $W$ for time $t$.  For any $0 < \eps < 1$, the snowflaked space $(H, d^{1-\eps})$ is also a metric space.  We will also work with the standard lattice
\begin{equation}\label{gamma-def}
 \Gamma \coloneqq \begin{pmatrix} 1 & \Z & \Z \\ 0 & 1 & \Z \\ 0 & 0 & 1 \end{pmatrix} = \{ [a,b,c]: a,b,c \in \Z \}
\end{equation}
which is a discrete cocompact subgroup of $H$, as well as the \emph{nilmanifold} $H/\Gamma$, which is a smooth compact three-dimensional manifold.

It was observed by Semmes \cite{semmes}, as a corollary of the work of Pansu \cite{pansu} on differentiation in Carnot groups, that there is no bilipschitz embedding of $(H,d)$ into a Hilbert space such as $\ell^2$ or $\R^D$.  On the other hand, just from the fact that $(H,d)$ is a doubling metric space, it follows from the work of Assouad \cite{assouad} that for any $0 < \eps \leq 1/2$, there is a bilipschitz embedding from $(H,d^{1-\eps})$ into $\ell^2$, with distortion\footnote{See Section \ref{notation-sec} for our conventions on asymptotic notation.} $O( \eps^{-1/2} )$.  In \cite{austin}, \cite{neiman} it was shown that this bound on the distortion is optimal in the case of the Heisenberg group, thus any bilipschitz map from $(H,d^{1-\eps})$ into $\ell^2$ or $\R^D$ must have distortion $\gtrsim \eps^{-1/2}$.  In \cite[Corollary 1.6]{austin} it was also shown that any bilipschitz map from the discrete ball $B_\Gamma(0,R) \coloneqq \{ \gamma \in \Gamma: d(0, \gamma) < R \}$ for $R \geq 2$ into $\ell^2$ or $\R^D$ must have distortion $\gtrsim \log^{1/2} R$, where $0 = [0,0,0]$ is the origin in $H$.  Some further explicit embeddings of $(H, d^{1-\eps})$ into $\ell^2$ (or $\ell^p$, $2 \leq p < \infty$) with optimal distortion were constructed in \cite{lee}, \cite{lafforgue}.

Assouad's construction allows the infinite dimensional space $\ell^2$ to be replaced with a finite-dimensional Euclidean space $\R^{D(\eps)}$, where $D(\eps)$ depends on $\eps$ (in fact one can take $D(\eps) = O(\eps^{-1})$).      
In \cite{neiman}, a bilipschitz embedding of $(H,d^{1-\eps})$ into $\R^D$ (again only using the doubling properties of $(H,d)$) with $D$ independent of $\eps$ was constructed; however, the distortion of this map was only bounded by $O( \eps^{-1 - c_D})$ rather than the optimal $O( \eps^{-1/2} )$, where $c_D \to 0$ as $D \to \infty$.

The main result of this paper is to show that one can in fact embed $(H,d^{1-\eps})$ into a bounded dimensional space with the optimal distortion:

\begin{theorem}\label{th1}  There exists a natural number $D$ such that for every $0 < \eps \leq 1/2$, there exists an embedding of $(H,d^{1-\eps})$ into $\R^D$ of distortion $O( \eps^{-1/2} )$.
\end{theorem}

This answers \cite[Question 3]{neiman} in the negative.  On the discrete ball $B_\Gamma(0,R)$ with $R \geq 2$, the metric $d$ is comparable to $d^{1-1/\log R}$, and hence we also obtain

\begin{corollary}\label{th1-cor}  There exists a natural number $D$ such that for every $R \geq 2$, there exists an embedding of $B_\Gamma(0,R)$ (with metric $d$) into $\R^D$ of distortion $O( \log^{1/2} R )$.
\end{corollary}

As mentioned previously, in \cite{austin} it was shown that this bound of $O(\log^{1/2} R)$ is optimal.

We now describe in informal terms the strategy of proof of Theorem \ref{th1}.  For a given $0 < \eps \leq 1/2$, our objective is to construct a map $\Phi \colon H \to \R^D$ into a bounded dimensional space $\R^D$ obeying the upper bound
\begin{equation}\label{phip}
 |\Phi(p) - \Phi(p')| \lesssim \eps^{-1/2} d(p,p')^{1-\eps}
\end{equation}
for all $p,p' \in H$, as well as the matching lower bound
\begin{equation}\label{pp'}
 |\Phi(p) - \Phi(p')| \gtrsim d(p,p')^{1-\eps}
\end{equation}
for all $p,p' \in H$.  Here and in the sequel we use $|v|$ to denote the norm of a vector $v \in \R^D$ in a Euclidean space.

By taking advantage of the freedom to increase the dimension $D$ by a bounded amount, it will suffice to obtain a map $\Phi_1 \colon H \to \R^{29}$ into a somewhat smaller\footnote{The dimension $29$ is what emerges from our specific form of the argument, but it is likely that this number could be reduced by a more careful analysis if desired.  More generally, we have not attempted to optimise the various numerical exponents in this paper, and one should not be too concerned with their precise values on a first reading.} dimensional space $\R^{29}$ that obeys the upper bound \eqref{phip} for all $p,p'$, and verifies the lower bound \eqref{pp'} just for a sparse range of distance scales $d(p,p')$, say when
$$ A^{n-0.1} \leq d(p,p') \leq 2 A^{n-0.1}$$
for an integer $n$ and a very large absolute constant $A>1$.  One can then hope to build $\Phi_1$ by a Weierstrass function type construction\footnote{To make the sum converge in the limit $n \to -\infty$ we will subtract off the constant $\phi_n(0)$ from $\phi_n(p)$ in the actual argument.}
$$ \Phi_1(p) \coloneqq \sum_{n = -\infty}^\infty A^{-\eps n} \phi_n(p)$$
where for each natural number $n$, $\phi_n \colon H \to \R^{29}$ is a function of magnitude $O(A^n)$ that ``oscillates'' at spatial scales $A^n$, analogously to the lacunary plane waves $x \mapsto A^n e^{i x / A^n}$ from $\R$ to $\C$ that one might see in the classical Weierstrass function construction; see Theorem \ref{th3} below for a precise statement.  This will establish the lower bound \eqref{pp'} by a Taylor expansion argument as long as the functions $\phi_{\geq n_0} := \sum_{n \geq n_0} A^{-\eps(n-n_0)} \phi_n(p)$ are quantitatively immersions in the sense that the wedge product 
$$ \bigwedge_{W =X,Y,Z} W \phi_{\geq n_0}(p) = X \phi_{\geq n_0}(p) \wedge Y \phi_{\geq n_0}(p) \wedge Z \phi_{\geq n_0}(p) \in \bigwedge^3 \R^{29}$$ 
is almost as large as possible; again, see Theorem \ref{th3} for a precise statement.

It remains to establish the upper bound.  This essentially amounts to obtaining good upper bounds for the magnitude of the gradient
$$ \sum_{n \geq n_0} A^{-\eps(n-n_0)} \nabla \phi_n(p)$$
uniformly in $n_0 \in \Z$ and $p \in H$, where $\nabla \phi \coloneqq (X \phi, Y \phi)$.  The triangle inequality will give a bound of the form $O(\eps^{-1})$, which corresponds roughly speaking to the results that of the bounds of Naor and Neiman \cite{neiman}, though we do not claim a simple way to reprove their results by our methods without going through most of the proof of Theorem \ref{th1}.  To improve this to $O(\eps^{-1/2})$, we will impose the orthogonality condition that each component $\nabla \phi_{n_0}(p)$ of the gradient is pointwise orthogonal to the more slowly varying function $\sum_{n > n_0} A^{-\eps(n-n_0)} \nabla \phi_n(p)$ for every $n_0$.  The desired bound of $O(\eps^{-1/2})$ will then easily follow from an induction argument and Pythagoras' theorem.

The main issue is then how to construct the functions $\phi_n$ in such a way that this orthogonality condition is satisfied.  After a rescaling, the problem reduces to one which can be informally described as follows: given a smooth, slowly varying function $\psi \colon H \to \R^{29}$ whose derivatives obey certain quantitative upper and lower bounds, construct a smooth oscillating function $\phi \colon H \to \R^{29}$, whose derivatives also obey certain quantitative upper and lower bounds, such that the bilinear form $B(\phi,\psi) \colon H \to \R^2$ defined by
\begin{equation}\label{bilinear-def}
 B(\phi,\psi) \coloneqq (W \phi \cdot W \psi)_{W = X,Y} = ( X \phi \cdot X \psi, Y \phi \cdot Y \psi )
\end{equation}
(with $\cdot  \colon \R^{29} \times \R^{29} \to \R$ being the usual dot product) vanishes identically:
\begin{equation}\label{b-eq}
B(\phi,\psi) = 0.
\end{equation}
We view this as an underdetermined system of differential equations for $\phi$ (two equations in $29$ unknowns).
The trivial solution $\phi=0$ to this equation will be inadmissible for our purposes due to the lower bounds we will require on $\phi$ (in order to obtain the quantitative immersion property mentioned previously, as well as for a stronger ``freeness'' property that is needed to close the iteration).  Because this construction will need to be iterated, it will be essential that the regularity control on $\phi$ is the same as that on $\psi$; one cannot afford to ``lose derivatives'' when passing from $\psi$ to $\phi$.  If one was embedding into an infinite dimensional space $\ell^2$, one could easily solve \eqref{b-eq} by ensuring that $\phi,\psi$ take values in orthogonal finite-dimensional subspaces of $\ell^2$; the difficulty is to solve this equation instead in the bounded dimensional setting of $\R^{29}$, in a fashion that allows for an indefinite amount of iteration.

This problem has some formal similarities with the isometric embedding problem \cite{gunther}, which can be viewed as the problem of constructing smooth solutions to an equation of the form $Q(\phi,\phi) = g$, where $(M,g)$ is a Riemannian manifold and $Q = (Q_{ij})_{1 \leq i,j \leq d}$ is the matrix-valued bilinear form
$$ Q(\phi,\psi)_{ij} = \partial_i \phi \cdot \partial_j \psi.$$
The isometric embedding problem also has the key obstacle that naive attempts to solve the equation $Q(\phi,\phi)=g$ iteratively can lead to an undesirable ``loss of derivatives'' that prevents one from iterating indefinitely.  This obstacle was famously 
resolved by the Nash-Moser iteration scheme \cite{nash}, \cite{moser} in which one alternates between perturbatively adjusting an approximate solution to improve the residual error term, and mollifying the resulting perturbation to counteract the loss of derivatives.  The current equation \eqref{b-eq} differs in some key respects from the isometric embedding equation $Q(\phi,\phi)=g$, in particular being linear in the unknown field $\phi$ rather than quadratic; nevertheless the key obstacle is the same, namely that naive attempts to solve either equation lose derivatives.  Our approach to solving \eqref{b-eq} will draw heavy inspiration\footnote{The iterative construction used here also bears some resemblance to the iterative construction used in Uchiyama's constructive proof \cite{uchiy} of the Fefferman-Stein decomposition theorem for functions of bounded mean oscillation.} from the Nash-Moser iteration technique, though it will not precisely use any of the standard forms of the Nash-Moser argument in the literature.

To motivate this iteration, we first express $B(\phi,\psi)$ using the product rule in a form that does not place derivatives directly on the unknown $\phi$:
\begin{equation}\label{b-expand}
 B(\phi,\psi) = \left( W(\phi \cdot W \psi) - \phi \cdot WW \psi\right)_{W = X,Y}
\end{equation}
This reveals that one can construct solutions $\phi$ to \eqref{b-eq} by solving the system of equations
\begin{equation}\label{phis}
\phi \cdot W \psi = \phi \cdot WW \psi =  0
\end{equation}
for $W \in \{X, Y \}$.  Because this system is zeroth order in $\phi$, this can easily be done by linear algebra (even in the presence of a forcing term $B(\phi,\psi)=F$) if one imposes a ``freeness'' condition (analogous to the notion of a free embedding in the isometric embedding problem) that $X \psi(p), Y \psi(p), XX \psi(p), YY \psi(p)$ are linearly independent at each point $p$, which one then adds to the list\footnote{For technical reasons, it will in fact be convenient to impose the stronger condition that the six vectors $X \psi(p)$, $Y \psi(p)$, $Z \psi(p)$, $XX \psi(p)$, $YY \psi(p)$, $XY \psi(p)$ are linearly independent.} of upper and lower bounds required on $\psi$ (with a related bound then imposed on $\phi$, in order to close the iteration).  However, as mentioned previously, there is a ``loss of derivatives'' problem with this construction: due to the presence of the differential operators $W$ in \eqref{phis}, a solution $\phi$ constructed by this method can only be expected to have two degrees less regularity than $\psi$ at best, which makes this construction unsuitable for iteration.

To get around this obstacle (which also prominently appears when solving (linearisations of) the isometric embedding equation $Q(\phi,\phi)=g$), we instead first construct a smooth, low-frequency solution $\phi_{\leq N_0} \colon H \to \R^{29}$ to a low-frequency equation
\begin{equation}\label{b-ortho}
B( \phi_{\leq N_0}, P_{\leq N_0} \psi ) = 0
\end{equation}
where $P_{\leq N_0} \psi$ is a mollification of $\psi$ (of Littlewood-Paley type) applied at a small spatial scale $1/N_0$ for some $N_0$, and then gradually relax the frequency cutoff $P_{\leq N_0}$ to deform this low frequency solution $\phi_{\leq N_0}$ to a solution $\phi$ of the actual equation \eqref{b-eq}.  

We will construct the low-frequency solution $\phi_{\leq N_0}$ rather explicitly, using the Whitney embedding theorem to construct an initial oscillating map $f$ into a very low dimensional space $\R^5$, composing it with a Veronese type embedding into a slightly larger dimensional space $\R^{20}$ to obtain a required ``freeness'' property, and then composing further with a slowly varying isometry $U(p) \colon \R^{20} \to \R^{29}$ depending on $P_{\leq N_0} \psi$ and constructed by a quantitative topological lemma (relying ultimately on the vanishing of the first few homotopy groups of high-dimensional spheres), in order to obtain the required orthogonality \eqref{b-ortho}; see Sections \ref{top-sec}, \ref{conclude} for details.

To perform the deformation of $\phi_{\leq N_0}$ into $\phi$, we must solve what is essentially the linearised equation
\begin{equation}\label{b-diff}
 B( \dot \phi, \psi ) + B( \phi, \dot \psi ) = 0
\end{equation}
of \eqref{b-eq} when $\phi$, $\psi$ (viewed as low frequency functions) are both being deformed at some rates $\dot \phi, \dot \psi$ (which should be viewed as high frequency functions).  To avoid losing derivatives, the magnitude of the deformation $\dot \phi$ in $\phi$ should not be significantly greater than the magnitude of the deformation $\dot \psi$ in $\psi$, when measured in the same function space norms.  For technical reasons, it will in fact be more convenient to use a discrete-time iteration rather than a continuous-time iteration, where the analogue of the time parameter is the Littlewood-Paley frequency parameter $N$ (or a logarithm thereof), but for sake of this informal discussion we will focus on the continuous-time equation \eqref{b-diff}.

As before, if one directly solves the difference equation \eqref{b-diff} using a naive application of \eqref{b-expand} with $B(\phi,\dot \psi)$ treated as a forcing term, one will lose at least one derivative of regularity when passing from $\dot \psi$ to $\dot \phi$.  However, observe that \eqref{b-expand} (and the symmetry $B(\phi, \dot \psi) = B(\dot \psi,\phi)$) can be used to obtain the identity
\begin{equation}\label{b-ident}
 B( \dot \phi, \psi ) + B( \phi, \dot \psi ) = \left( W(\dot \phi \cdot W \psi + \dot \psi \cdot W \phi) - (\dot \phi \cdot WW \psi + \dot \psi \cdot WW \phi)\right)_{W = X,Y}
\end{equation}
and then one can solve \eqref{b-diff} by solving the system of equations
$$ \dot \phi \cdot W \psi = - \dot \psi \cdot W \phi$$
for $W \in \{X,XX,Y,YY\}$.  The key point here is that this system is zeroth order in both $\dot \phi$ and $\dot \psi$, so one can solve this system without losing any derivatives when passing from $\dot \psi$ to $\dot \phi$; compare this situation with that of the superficially similar system
$$ \dot \phi \cdot W \psi = - \phi \cdot W \dot \psi $$
that one would obtain from naively linearising \eqref{phis} without exploiting the symmetry\footnote{This symmetry exploiting trick however comes with a cost: we were unable to use this scheme to also impose the orthogonality conditions $X \phi \cdot Y \psi = 0$ and $X \psi \cdot Y \phi = 0$, which would otherwise have been quite useful in ensuring that the function $\psi$ retains the required freeness and immersion properties upon iteration; this is because each of these equations fails to be symmetric on $\phi$ and $\psi$.  Instead, we will have to perform a delicate analysis of how the wedge product $X \psi \wedge Y \psi$ evolves as one replaces $\psi$ with $\psi+\phi$, relying in particular on a careful computation of components of a certain pseudoinverse matrix.}
 of $B$.  There is still however one residual ``loss of derivatives'' problem arising from the presence of a differential operator $W$ on the $\phi$ term, which prevents one from directly evolving this iteration scheme in time without losing regularity in $\phi$.  It is here that we borrow the final key idea of the Nash-Moser scheme, which is to replace $\phi$ by a mollified version $P_{\leq N} \phi$ of itself (where the frequency scale $N$ of the projection $P_{\leq N}$ depends on the time parameter).  This creates an error term in \eqref{b-diff}, but it turns out that this error term is quite small and smooth (being a ``high-high paraproduct'' of $\nabla \phi$ and $\nabla\psi$, it ends up being far more regular than either $\phi$ or $\psi$, even with the presence of the derivatives) and can be iterated away provided that the initial frequency cutoff $N_0$ is large and the function $\psi$ has a fairly high (but finite) amount of regularity (we will eventually use the H\"older space $C^{20,\alpha}$ to measure this).

It seems likely that this method can extend to other Carnot groups than $H$, and perhaps even to arbitrary nilpotent Lie groups.  Certainly the case of Carnot groups of nilpotency class $2$ (such as higher dimensional Heisenberg groups) should follow by a straightforward adaptation of the arguments in this paper.  However, we will not pursue these generalisations here.

\begin{remark} We briefly discuss\footnote{We thank Assaf Naor for these observations, and an anonymous commenter on the author's blog for pointing out the breakdown of the arguments in this paper for $\eps$ close to $1$.} the situation of extremely snowflaked metrics $(H,d^{1-\eps})$ with $1/2 < \eps \leq 1$.  Such spaces have Hausdorff dimension $\frac{4}{1-\eps}$ and so can only be embedded in a fashion bilipschitz into $\R^D$ if $D \geq \frac{4}{1-\eps}$, regardless of distortion.  On the other hand, by first embedding $(H, d^{1/2})$ in a bilipschitz fashion into $\R^D$ for a fixed $D$ (either by the results of this paper, or earlier results such as \cite{neiman}) and then embedding $(\R^D, d_{\R^D}^{2(1-\eps)})$ with bounded distortion into $\R^{D'}$ with $D' = O( \frac{D}{1-\eps})$ using the constructions in \cite{kahane}, \cite{talagrand}, one can embed $(H,d^{1-\eps})$, $1/2 < \eps \leq 1$ with bounded distortion into a Euclidean space of dimension $O( \frac{1}{1-\eps})$.
\end{remark}

\section{Acknowledgments}

We thank Assaf Naor for suggesting this problem and providing many useful comments.  We are indebted to the anonymous referee for an extremely thorough reading of the manuscript and many useful corrections and suggested improvements.

\section{Notation}\label{notation-sec}

If $\R^{D_1}, \R^{D_2}$ are vector spaces the space of linear maps from $\R^{D_1}$ to $\R^{D_2}$ will be identified with $\R^{D_1 D_2}$ in the obvious fashion, with the Euclidean norm on the latter being the Frobenius norm on the former.  Thus if $T \colon \R^{D_1} \to \R^{D_2}$ is a linear map, $|T|$ will denote its Frobenius norm.

In a similar vein, the exterior power $\bigwedge^k \R^D$ of a vector space $\R^D$ with $1 \leq k \leq D$ can be identified with $\R^{\binom{D}{k}}$ (with orthonormal basis $e_{i_1} \wedge \dots \wedge e_{i_k}$ with $1 \leq i_1 < \dots < i_k \leq D$, and in particular the Euclidean norm on the former is inherited by the latter.  We observe the \emph{depolarised Cauchy-Binet formula}
\begin{equation}\label{depolar}
\langle u_1 \wedge \dots \wedge u_k, v_1 \wedge \dots \wedge v_k \rangle = \det (u_i \cdot v_j)_{1 \leq i,j \leq k},
\end{equation}
where $\langle, \rangle$ denotes the inner product on $\bigwedge^k \R^D$ associated to the above Euclidean structure; this is easily verified from multilinearity by checking the case when all of the $u_i,v_j$ are drawn from the standard basis $e_1,\dots,e_D$.  Specialising to the case $u_i=v_i$, we obtain the more traditional \emph{Cauchy-Binet formula}
\begin{equation}\label{cauchy-binet}
 \det(TT^*) = \det(v_i \cdot v_j)_{1 \leq i,j \leq k} = |v_1 \wedge \dots \wedge v_k|^2.
\end{equation}
for any $v_1,\dots,v_k \in \R^D$, where $T \colon \R^D \to \R^k$ denotes the linear map
$$ T(u) \coloneqq (u \cdot v_1, \dots, u \cdot v_k).$$
This identity makes quantitative the standard fact that $T$ is full (row) rank if and only if $v_1,\dots,v_k$ are linearly independent.  For instance, the $k=2$ case of \eqref{cauchy-binet} is the Lagrange identity
$$ |v_1 \wedge v_2|^2 = |v_1|^2 |v_2|^2 - |v_1 \cdot v_2|^2.$$

We use the asymptotic notation $A \lesssim B$ or $A = O(B)$ to denote the bound $|A| \leq CB$ for a constant $C$, and write $A \sim B$ for $A \lesssim B \lesssim A$.  If we need $C$ to depend on parameters, we will indicate this by subscripts, for instance $A \lesssim_{C_0, N_0} B$ denotes an estimate of the form $|A| \leq C(C_0,N_0) B$ where the implied constant $C(C_0,N_0)$ depends only on $C_0$ and $N_0$.  This notation will be extended to linear maps or elements of the exterior algebra using the norms indicated above.

If $\phi \colon H \to \R^D$ is a smooth function, we let $\nabla \phi \colon H \to \R^{2D}$ denote the Heisenberg gradient
$$ \nabla \phi \coloneqq (X \phi, Y \phi);$$
iterating this, we have $\nabla^k \phi \colon  H \to \R^{2^k D}$ for any $k \geq 1$.  We then define the $C^0$ norm 
$$ \| \phi \|_{C^0} \coloneqq \sup_{p \in H} |\phi(p)|$$
and more generally the $C^k$ norm
$$ \| \phi \|_{C^k} \coloneqq \sum_{0 \leq j \leq k} \| \nabla^j \phi \|_{C^0}$$
for any natural number $k$; more generally, for any spatial scale $R>0$, we define the $C^k_R$ norm
$$ \| \phi \|_{C^k_R} \coloneqq \sum_{0 \leq j \leq k} R^j \| \nabla^j \phi \|_{C^0}$$
which is a rescaled version of the $C^k$ norm that is adapted to the spatial scale $R$.
For technical reasons we will eventually need to work with H\"older spaces (which are better behaved with respect to Littlewood-Paley decompositions than more classical spaces such as $C^k$).  We fix a H\"older exponent $0 < \alpha < 1$ (e.g., one can take $\alpha \coloneqq 1/2$ throughout this paper), and allow all implied constants to depend on $\alpha$.  We define the homogeneous H\"older norm
$$ \| \phi \|_{\dot C^{0,\alpha}} \coloneqq \sup_{p,q \in H: p \neq q} \frac{|\phi(p)-\phi(q)|}{d(p,q)^\alpha},$$
defined (though possibly infinite) for all smooth $\phi \colon  H \to \R^D$.  We then define the higher H\"older norms
$$ \| \phi \|_{C^{k,\alpha}} \coloneqq \| \phi \|_{C^k} + \| \nabla^k \phi \|_{\dot C^{0,\alpha}}$$
for $k \geq 0$ and smooth $\phi \colon  H \to \R^D$, and more generally define the rescaled H\"older norms
$$ \| \phi \|_{C^{k,\alpha}_R} \coloneqq \| \phi \|_{C^k_R} + R^{k+\alpha} \| \nabla^k \phi \|_{\dot C^{0,\alpha}}$$
for any $k \geq 0$ and $R>0$, and smooth $\phi \colon  H \to \R^D$.

By many applications of the product rule, one can verify the algebra properties
\begin{equation}\label{algebra-1}
 \| \phi \psi\|_{C^{k}_R} \lesssim_k \| \phi \|_{C^{k}_R} \| \psi \|_{C^{k}_R}.
\end{equation}
and
\begin{equation}\label{algebra-2} \| \phi \psi\|_{C^{k,\alpha}_R} \lesssim_k \| \phi \|_{C^{k,\alpha}_R} \| \psi \|_{C^{k,\alpha}_R}.
\end{equation}
for any smooth $\phi, \psi \colon H \to \R$, $k \geq 0$, and $R>0$.  Similarly if $\phi, \psi$ are vector-valued instead of scalar-valued, and one forms the wedge product or dot product instead of the pointwise product; observe that the implied constants here will not depend on the dimension of the vector space that $\phi$ or $\psi$ ranges in (because the Cauchy-Schwarz inequalities for such products do not contain dimension-dependent constants).

For any $\lambda > 0$, define the scaling maps $\delta_\lambda \colon H \to H$ by
$$ \delta_\lambda[x,y,z] \coloneqq [\lambda x, \lambda y, \lambda^2 z];$$
these are automorphisms of $H$ that obey the scaling law
\begin{equation}\label{scaling}
d(\delta_\lambda(p), \delta_\lambda(p')) = \lambda d(p,p')
\end{equation}
for all $p,q \in H$, as well as the chain rules
\begin{equation}\label{chain}
\begin{split}
X ( \phi \circ \delta_\lambda) &= \lambda (X\phi) \circ \delta_\lambda \\
Y ( \phi \circ \delta_\lambda) &= \lambda (Y\phi) \circ \delta_\lambda \\
Z ( \phi \circ \delta_\lambda) &= \lambda^2 (Z\phi) \circ \delta_\lambda 
\end{split}
\end{equation}
for any smooth $\phi \colon  H \to \R^D$.  One can think of $X,Y$ as being ``first-order'' with respect to this scaling family $(\delta_\lambda)_{\lambda>0}$, while $Z=YX-XY$ should be thought of as being ``second-order'', despite being a first-order differential operator.  
From iterating \eqref{chain} we have
\begin{equation}\label{chain-iter}
 \nabla^k (\phi \circ \delta_\lambda) = \lambda^k (\nabla^k \phi) \circ \delta_\lambda
\end{equation}
for any $k \geq 1$ and $\lambda>0$, and any smooth $\phi \colon  H \to \R^D$.  

From \eqref{chain-iter}, \eqref{scaling} one observes the scaling laws
\begin{equation}\label{scaling-1}
 \| \phi \circ \delta_\lambda \|_{C^k_R} = \| \phi \|_{C^k_{\lambda R}}
\end{equation}
and
\begin{equation}\label{scaling-2}
 \| \phi \circ \delta_\lambda \|_{C^{k,\alpha}_R} = \| \phi \|_{C^{k,\alpha}_{\lambda R}}
\end{equation}
for all smooth $\phi \colon  H \to \R^D$, $k \geq 0$, $\lambda>0$, and $R>0$.

A \emph{dyadic number} is a number of the form $2^n$, where $n$ is an integer; these are the scales we will use for Littlewood-Paley decompositions, which we discuss in Section \ref{lp-sec}.

\section{Reduction to constructing a lacunary family of oscillating functions}

In this section we reduce Theorem \ref{th1} to the task of finding a family of functions $\phi_n$ that oscillate at different scales $A^n$, and obey an orthogonality condition.

For the rest of the paper, we select absolute constants in the following order:
\begin{itemize}
\item A sufficiently large natural number $C_0>1$.   (This is a general-purpose constant used to make explicit the bounds in certain inductive hypotheses.)
\item A sufficiently large dyadic number $N_0$ (depending on $C_0$).  (This is a large frequency scale at which we initialise a certain Nash-Moser type iteration.)
\item A sufficiently large dyadic number $A$ (depending on $C_0, N_0$).  (This very large quantity controls the sparsity of a certain family of scales that we will control in our construction.)
\end{itemize}
Observe that any quantity depending on earlier quantities in this hierarchy can be bounded by quantities later in this hierarchy; for instance, if $Q$ is a quantity depending on $C_0$ and $N_0$, then we have $Q \lesssim_{N_0} 1$ and $Q \leq \log\log A$.  We will use these sorts of manipulations in the sequel without further comment.

To show Theorem \ref{th1}, we may assume the technical condition
$$ \eps \not \in [1/A, 1/\log^{1/2} A]$$
since if $\eps$ falls into this interval, we may simply replace $A$ by (say) $e^{A^2}$ to avoid this range.

It will suffice to establish the Lipschitz lower bound on a sparse set of scales, namely it suffices to construct (for each $0 < \eps \leq 1/2$ avoiding $[1/A, 1/\log^{1/2} A]$) a map $\Phi_1 \colon H \to \R^{29}$ obeying the Lipschitz upper bound
\begin{equation}\label{nup}
 | \Phi_1(p) - \Phi_1(p')| \lesssim_{A} \eps^{-1/2} d(p,p')^{1-\eps} 
\end{equation}
for all $p,p' \in H$, and the Lipschitz lower bound
\begin{equation}\label{nlp} 
| \Phi_1(p) - \Phi_1(p') | \gtrsim_{A} d(p,p')^{1-\eps}
\end{equation}
whenever $p,p' \in H$ are such that $A^{n_0-0.1} \leq d(p,p') \leq 2 A^{n_0-0.1}$ for some integer $n_0$.  Indeed, suppose that such a map has been constructed.  Then if we write $A = 2^M$, one can easily verify using \eqref{scaling} that the map $\phi \colon  H \to \R^{29M}$ defined by
\begin{equation}\label{Phi-def}
 \Phi(p) \coloneqq \left( \Phi_1( \delta_{2^m}(p) )\right)_{m=0}^{M-1}
\end{equation}
obeys the upper bound
$$ | \Phi(p) - \Phi(p')| \lesssim_{A} \eps^{-1/2} d(p,p')^{1-\eps} $$
and the lower bound
$$ | \Phi(p) - \Phi(p')| \gtrsim_{A} d(p,p')^{1-\eps} $$
for all $p,p' \in H$, thus giving Theorem \ref{th1} (after choosing the parameters $C_0,N_0,A$, and setting $D \coloneqq 29M$).

To construct the map $\Phi_1$, we construct the following family of oscillating functions.

\begin{theorem}[Maps oscillating at lacunary scales]\label{th3}  Let $0 < \eps \leq 1/2$ avoid the range $[1/A,1/\log^{1/2} A]$.  Then one can find a smooth map $\phi_n \colon H \to \R^{29}$ for each integer $n$ obeying the following bounds:
\begin{itemize}
\item (Smoothness at scale $A^n$) For all integers $n$, one has
\begin{equation}\label{smooth}
\| \phi_n \|_{C^6_{A^n}} \lesssim_{C_0} A^n.
\end{equation}
In particular, we have 
\begin{equation}\label{nabla}
X \phi_n(p), Y \phi_n(p) = O_{C_0}(1); \quad Z \phi_n(p) = O_{C_0}(A^{-n})
\end{equation}
for all $p \in H$.
\item (Orthogonality)  If $\eps \leq 1/A$, then for all integers $n$, one has
\begin{equation}\label{ortho-1}
\sum_{n' > n} A^{-\eps(n'-n)} B( \phi_n, \phi_{n'} ) = 0
\end{equation}
identically on $H$, where $B$ is the bilinear form \eqref{bilinear-def}.  Note that the sum in \eqref{ortho-1} converges absolutely thanks to \eqref{nabla}.
\item (Non-degeneracy and immersion)  
For all integers $n$ and all $p \in H$, one has
\begin{equation}\label{nondeg-1}
|X \phi_n(p)|, |Y \phi_n(p)| \gtrsim_{C_0} 1
\end{equation}
and 
\begin{equation}\label{nondeg}
 \left|\bigwedge_{W = X,Y,Z} \left(\sum_{n' \geq n} A^{-\eps(n'-n)} W \phi_{n'}(p)\right)\right| \gtrsim_{C_0} A^{-n} \sum_{n' \geq n} A^{-2\eps(n'-n)} 
\end{equation}
Again, the sums in \eqref{nondeg} converge absolutely thanks to \eqref{nabla}. 
\end{itemize}
\end{theorem}

We will establish Theorem \ref{th3} in later sections.  For now, let us assume it and show how it can be used to construct a function $\Phi_1 \colon H \to \R^D$ obeying the desired properties \eqref{nup}, \eqref{nlp}.

Fix $0 < \eps \leq 1/2$ avoiding the range $[1/A,1/\log^{1/2} A]$.  We construct $\Phi_1$ by the explicit formula
$$ \Phi_1(p) \coloneqq \sum_{n=-\infty}^\infty A^{-\eps n} (\phi_n(p) - \phi_n(0))$$
where $0 = [0,0,0]$ is the origin in $H$. Observe from \eqref{smooth}, \eqref{nabla} that one has the bounds
\begin{equation}\label{flip}
|\phi_n(p) - \phi_n(p')| \lesssim_{C_0} \min( A^n, d(p,p') )
\end{equation}
for any $p,p' \in H$, so the sum here is locally uniformly absolutely convergent.

Now we establish the upper bound \eqref{nup}.  We may assume $A^{n_0-1} \leq d(p,p') \leq A^{n_0}$ for some integer $n_0$.
By applying the rescaling $\delta_{A^{n_0}}$ (replacing each $\phi_n$ with $A^{-n_0} \phi_{n+n_0} \circ \delta_{A^{n_0}}$) we may assume without loss of generality that $n_0=0$; similarly, by translating by $p'$ (and subtracting $\Phi_1(p')$ from $\Phi_1$) we may assume $p'=0$.  Thus
$$ A^{-1} \leq d(p,0) \leq 1$$
and it will suffice to establish the bound
$$ |\Phi_1(p)| \lesssim_{A} \eps^{-1/2}.$$
If we introduce the low frequency component
$$ \Psi(p) := \sum_{n=0}^\infty A^{-\eps n} (\phi_n(p) - \phi_n(0))$$
of $\Phi_1$, then from \eqref{flip} and the triangle inequality we have
\begin{equation}\label{pis}
 \Phi_1(p) = \Psi(p) + O_{C_0}( A^{\eps-1} )
\end{equation}
so it will suffice to show that
\begin{equation}\label{pp}
 |\Psi(p)| \lesssim_{A} \eps^{-1/2}.
\end{equation}
From \eqref{smooth} we have
$$ \nabla^j( \phi_n(p) - \phi_n(0) ) \lesssim_{C_0} A^{-(j-1)n}$$
for $1 \leq j \leq 6$, so from this (and \eqref{flip}) the sum for $\nabla \Psi$ converges in the $C^5$ topology, and from the triangle inequality one has the bounds
\begin{equation}\label{eps-1}
 \| \nabla \Psi \|_{C^0} \lesssim_{C_0} \sum_{n=0}^\infty A^{-\eps n}
\end{equation}
and
\begin{equation}\label{c6}
 \| \nabla^j \Psi \|_{C^0} \lesssim_{C_0} 1
\end{equation}
for $2 \leq j \leq 6$.  Actually, we claim the crucial improvement
\begin{equation}\label{eps-2}
 \| \nabla \Psi \|_{C^0} \lesssim_{A} M
\end{equation}
to \eqref{eps-1}, that is to say that
$$ \left|\sum_{n \geq 0} A^{-\eps n} \nabla \phi_n(p)\right| \lesssim_{C_0} M$$
for any $p$, where $M \sim_A \eps^{-1/2}$ is the quantity defined by
$$ M \coloneqq \left(\sum_{n \geq 0} A^{-2\eps n}\right)^{1/2}.$$
For $\eps > 1/A$, the claim already follows from \eqref{eps-1}, as the right hand side of this estimate is now comparable to $1$.  Thus we may assume $\eps \leq 1/A$, so that \eqref{ortho-1} holds.  From this equation and Pythagoras' theorem one has
$$ \left|\sum_{n \geq m} A^{-\eps n} \nabla \phi_n(p)\right|^2 = \left|\sum_{n \geq m+1} A^{-\eps n} \nabla \phi_n(p)\right|^2 + A^{-2\eps m} \left|\nabla \phi_m(p)\right|^2$$
for any $m$, which telescopes to the Bessel type equality
$$ \left|\sum_{n \geq 0} A^{-\eps n} \nabla \phi_n(p)\right|^2 = \sum_{n \geq 0} A^{-2\eps n} |\nabla \phi_n(p)|^2$$
and the claim \eqref{eps-2} now follows from \eqref{nabla}.  From \eqref{eps-2} and the fundamental theorem of calculus (noting that $\Psi(0)=0$ and $M = O(\eps^{-1/2})$) we obtain \eqref{pp} as required.  For future reference, we observe that this argument, when combined with \eqref{nondeg-1} also gives matching lower bounds in the case $\eps \leq 1/A$, so that
\begin{equation}\label{match}
|X \Psi(p)|, |Y \Psi(p)| \sim_{C_0} M
\end{equation}
for all $p \in H$.

Now we prove \eqref{nlp}.  Let $p,p' \in H$ be such that $A^{n_0-0.1} \leq d(p,p') \leq 2 A^{n_0-0.1}$ for some integer $n_0$.  As before we may normalise $n_0=0$ and $p'=0$, thus
\begin{equation}\label{ao}
 A^{-0.1} \leq d(p,0) \leq 2 A^{-0.1}
\end{equation}
and it will suffice to establish the bound
$$ |\Phi_1(p)| \gtrsim_{C_0} A^{-0.2}.$$
By \eqref{pis} it suffices to obtain the bound
\begin{equation}\label{psil} 
|\Psi(p)| \gtrsim_{C_0} A^{-0.2}.
\end{equation}

We estimate some derivatives of $\Psi$ in preparation for performing a Taylor expansion.  By construction, $\Psi(0)=0$.
From \eqref{c6} and \eqref{heisenberg} one has
\begin{equation}\label{xyz0}
 |Z \Psi(0)|, |W_1 W_2\Psi(0)|, |W_1 W_2 W_3\Psi(p)| \lesssim_{C_0} 1
\end{equation}
for all $p \in H$ and $W_1,W_2,W_3 \in \{X,Y,Z\}$.
Also, from \eqref{nondeg} we have
\begin{equation}\label{xyz-lower}
 |X \Psi(0) \wedge Y \Psi(0) \wedge Z \Psi(0)| \gtrsim_{C_0} M^2.
\end{equation}
By Cauchy-Schwarz, this also implies
\begin{equation}\label{xyz-lower-2}
 |X \Psi(0) \wedge Y \Psi(0) | \gtrsim_{C_0} M^2.
\end{equation}
Write $p = \exp( x X + y Y + z Z)(0)$ for some $x,y,z \in \R$; from \eqref{ao} and \eqref{scaling} we see that $(A^{0.1} x, A^{0.1} y, A^{0.2} z)$ is comparable in magnitude to one.  By Taylor expansion and \eqref{xyz0}, we thus have
$$ \Psi(p) = (x X + y Y + z Z)\Psi(0) + \frac{1}{2} (xX + y Y + zZ)^2 \Psi(0) + O_{C_0}( A^{-0.3} ).$$
If $|x| \geq A^{-0.15}$ or $|y| \geq A^{-0.15}$, we simplify the above expansion to
$$ \Psi(p) =  (x X + y Y)\Psi(0) + O_{C_0}(A^{-0.2})$$
and then from \eqref{match}, \eqref{xyz-lower-2} we will have
\begin{equation}\label{room}
 |\Psi(p)| \gtrsim_{C_0} M A^{-0.15}
\end{equation}
which is acceptable with substantial room to spare (since $M \geq 1$).  Now suppose that $|x|, |y| \leq A^{-0.15}$, which forces $|z| \sim A^{-0.2}$.  Then we simplify the above Taylor expansion to
$$ \Psi(p) = (x X + y Y + z Z)\Psi(0) + O_{C_0}( A^{-0.3} )$$
and hence the orthogonal projection of $\Psi(p)$ to the subspace of $\R^{29}$ orthogonal to $X \Psi(0)$ and $Y \Psi(0)$ has norm $\gtrsim_{C_0} |z| \sim_{C_0} A^{-0.2}$, thanks to \eqref{xyz-lower} and \eqref{match}.  Thus in either case we obtain the desired bound \eqref{psil}.

It remains to prove Theorem \ref{th3}.  This will be the objective of the remaining sections of the paper.

\begin{remark}  The fact that there is room to spare in \eqref{room} indicates that one can make tighter estimates\footnote{We thank Assaf Naor for raising this possibility.}.  Indeed, an inspection of the above argument reveals that whenever $|x| \geq A^{-0.15}/M$ or $|y| \geq A^{-0.15}/M$, one has
$$ |\Psi(p)| \gtrsim_{C_0} M |x| + M |y|.$$
Using this more refined estimate, one can eventually establish the lower bound
$$ |\Phi_1(p) - \Phi_1(q)| \gtrsim_A F_\eps( p q^{-1} ) $$
whenever $A^{n-0.1} \leq d(p,q) \leq 2A^{n-0.1}$ for some integer $n$, 
where $F_\eps: H \to \R$ is the function
$$ F_\eps([x,y,z]) \coloneqq M |x|^{1-\eps} + M |y|^{1-\eps} + |z|^{(1-\eps)/2};$$
similar arguments also give the matching upper bound
$$ |\Phi_1(p) - \Phi_1(q)| \lesssim_A F_\eps( p q^{-1} ) $$
for all $p,q \in H$.  As a consequence, the function $\Phi$ defined by \eqref{Phi-def} in fact enjoys the estimates
\begin{equation}\label{phipq}
 |\Phi(p) - \Phi(q)| \sim_A F_\eps(p q^{-1})
\end{equation}
for all $p,q \in H$.  Note that this is stronger than Theorem \ref{th1} since
$$ d(p,q)^{1-\eps} \lesssim F_\eps(pq^{-1}) \lesssim M d(p,q)^{1-\eps}$$
for all $p,q \in H$.  We leave the detailed verifications of these claims to the interested reader.  An embedding of $H$ into $\ell^2$ that also obeyed the estimate \eqref{phipq} was previously obtained in \cite{lee}.
\end{remark}

\section{Reduction to the iterative step}

Theorem \ref{th3} will be established by iterating the following proposition.
Because we need to use this proposition in an inductive argument, it will be important that we avoid using asymptotic notation such as $\lesssim$ in the \emph{hypotheses} of the proposition, though we will continue to use this notation in its \emph{conclusions}.

\begin{proposition}[Key iterative step]\label{induct}  Let $M$ be a quantity with
\begin{equation}\label{xp-large}
M \geq C_0^{-1}.
\end{equation}
Suppose one has a smooth map $\psi \colon H \to \R^{29}$ obeying the following estimates:
\begin{itemize}
\item[(i)]  (Non-degenerate first derivatives)  For any $p \in H$, one has
\begin{align}
C_0^{-1} M \leq |X \psi(p)|, |Y \psi(p)| &\leq C_0 M. \label{xyp-comp}\\
|X \psi(p) \wedge Y \psi(p)| &\geq C_0^{-6} M^2. \label{xyp-nondeg}
\end{align}
\item[(ii)]  (Locally free embedding)  For any $p \in H$, one has
\begin{equation}\label{free}
\left|\bigwedge_{W = X,Y,Z,XX,YY,XY} W \psi(p)\right| \geq C_0^{-20} A^{-4} M^2.
\end{equation}
\item[(iii)]  (H\"older regularity at scale $A$)  One has
\begin{equation}\label{j-hold}
 \| \nabla^2 \psi \|_{C^{18,\alpha}_A} \leq C_0^2 A^{-1}.
\end{equation}
\end{itemize}
Then one can find a smooth map $\phi \colon  H \to \R^{29}$ obeying the following estimates.
\begin{itemize}
\item[(iv)] (Non-degenerate first derivatives)  For any $p \in H$, one has
\begin{equation}\label{xyp}
 |X \phi(p)|, |Y \phi(p)| \gtrsim 1
\end{equation}
and
\begin{equation}\label{more-ortho}
|X(\psi+\phi)(p) \wedge Y(\psi+\phi)(p)|^2 - |X\psi(p) \wedge Y\psi(p)|^2 \gtrsim C_0^{-4} M^2
\end{equation}
(in particular, the left-hand side of \eqref{more-ortho} is non-negative).
\item[(v)]  (Locally free embedding)  For any $p \in H$, one has
\begin{equation}\label{free-inc}
 \left|\bigwedge_{W = X,Y,Z,XX,YY,XY} W(\psi+\phi)(p)\right| \gtrsim C_0^{-12} M^2.
\end{equation}
\item[(vi)]  (H\"older regularity at scale $1$)  One has
\begin{equation}\label{j-hold-phi}
 \| \phi \|_{C^{20,\alpha}} \lesssim 1.
\end{equation}
\item[(vii)] (Orthogonality)  We have
\begin{equation}\label{x-ortho}
B( \phi, \psi ) = 0.
\end{equation}
\end{itemize}
\end{proposition}

The hypotheses and conclusions here are technical, chosen so that one can close a certain induction argument.  In particular it will be crucial that the function $\phi$ has essentially the same sort of regularity control (in this case, $C^{20,\alpha}$ type control) that the original function $\psi$ has; one cannot afford to ``lose derivatives'' in this regard.  It is because of this that we will be forced to use a version of the Nash-Moser iteration scheme to construct $\phi$.   On the other hand, the condition \eqref{j-hold} ensures that the higher derivatives of the given function $\psi$ are quite small, gaining one or more powers of the large quantity $A$, and these factors will be essential in allowing one to keep the constants in the conclusions of Proposition \ref{induct} at a manageable level, and in particular to be able to close the induction.  The freeness property in \eqref{free-inc} is stronger than what is needed to establish the immersion property \eqref{nondeg}, but will be important for inductive purposes, as it is needed for the Nash-Moser style argument to work.  The powers of $C_0$ in the conclusions of Proposition \ref{induct} are superior to those in the hypotheses, which is needed to close the induction; we will be able to obtain these gains due to the very slowly varying nature of $\psi$, as represented by the appearance of the large parameter $A$ in the hypotheses.

We establish Proposition \ref{induct} in later sections.  In this section, we show how Proposition \ref{induct} can be iterated to establish Theorem \ref{th3}.

We first construct an auxiliary function $\phi^0 \colon H \to \R^{20}$ in a slightly lower dimensional Euclidean space $\R^{20}$ than $\R^{29}$, which essentially allows one to verify Theorem \ref{th3} for a single scale $n$, and will also be useful in later sections for inductively increasing the range of $n$ for which Theorem \ref{th3} can be verified. 

\begin{proposition}[A single oscillating function]\label{prop}  There exists a smooth map $\phi^0 \colon H \to \R^{20}$ obeying the following estimates:
\begin{itemize}
\item (Smoothness)  For any non-negative integer $j$, we have
\begin{equation}\label{phi0-smooth}
\| \phi^0 \|_{C^j} \lesssim_j 1.
\end{equation}
\item (Locally free embedding)  For any $p \in H$, we have
\begin{equation}\label{free-embed}
 \left|\bigwedge_{W = X,Y,Z,XX,YY,XY} W \phi^0(p)\right| \gtrsim 1.
\end{equation}
\end{itemize}
In particular, from Cauchy-Schwarz we also derive the estimates
$$ |W_1 \phi^0(p) \wedge \dots \wedge W_k \phi^0(p)| \gtrsim 1$$
whenever $p \in H$ and $W_1,\dots,W_k$ are distinct differential operators in $\{ X, Y, Z, XX, YY, XY \}$.  (From \eqref{heisenberg} one can also replace $XY$ by $YX$ in this latter claim.)
\end{proposition}

The dimension $20$ in this proposition can almost certainly be lowered, but we have not attempted to optimise it here.

\begin{proof} 
As mentioned in the introduction the nilmanifold $H/\Gamma$ is smooth compact three-dimensional manifold.  By the strong Whitney immersion\footnote{We thank the referee for pointing out that this theorem improves the numerical dimensions from the previous argument of the author which relied instead on the Whitney embedding theorem.} theorem \cite{whitney}, there is a smooth immersion of $H/\Gamma$ into $\R^{2 \times 3 - 1} = \R^5$, which lifts to a smooth map $f \colon H \to \R^5$ which is $\Gamma$-automorphic in the sense that $f(p\gamma) = f(p)$ for all $p \in H$ and $\gamma \in \Gamma$.  We fix this map $f$ (in particular we may allow implied constants to depend on $f$).  By compactness of $H/\Gamma$, we have
$$ \| f \|_{C^j} \lesssim_j 1$$
for every $j$.  The vector fields $X,Y,Z$ push forward to pointwise linearly independent vector fields on the compact manifold $H/\Gamma$, and hence we have
$$ |X f(p) \wedge Y f(p) \wedge Z f(p)| \gtrsim 1$$
for all $p \in H$.

This does not quite recover the full strength of \eqref{free-embed}. To do this, we perform the trick (standard in the Nash embedding theorem literature) of composing $f$ with the Veronese-type embedding $V \colon \R^5 \to \R^5 \times \mathrm{Sym}^2(\R^5) \equiv \R^{5 + \frac{5(5+1)}{2}} = \R^{20}$ defined by
$$ V(v) \coloneqq (v, v \otimes v)$$
where $\otimes \colon \R^5 \times \R^5 \to \R^{5} \otimes \R^5$ is the tensor product, and $\mathrm{Sym}^2(\R^5) \subset \R^5 \otimes \R^5$ is the subspace of symmetric rank $2$ tensors.  Let $\phi^0 \colon H \to \R^{20}$ be the map $\phi^0 \coloneqq V \circ f$, thus
$$ \phi^0(p) = ( f(p), f(p) \otimes f(p)).  $$
From the chain rule (or product rule) we certainly have \eqref{phi0-smooth}.  Now suppose that there is a point $p \in H$ for which the quantity
\begin{equation}\label{xyzphi}
 \bigwedge_{W = X,Y,Z,XX,YY,XY} W \phi^0(p)
\end{equation}
vanishes, thus we have a non-trivial linear dependence
$$ \sum_{W = X,Y,Z,XX,YY,XY} a_W W \phi^0(p) = 0$$
for some real numbers $a_W$ for $W = X,Y,Z,XX,YY,XY$, not all zero.  For brevity we omit dependence on $p$.  In components, this means that
\begin{equation}\label{linear}
 \sum_{W = X,Y,Z,XX,YY,XY} a_W W f= 0
\end{equation}
and
\begin{equation}\label{quadratic}
 \sum_{W = X,Y,Z,XX,YY,XY} a_W W (f \otimes f) = 0.
\end{equation}
Taking the tensor product of \eqref{linear} with $f$ on the left and the right and subtracting from \eqref{quadratic} using the product rule, we conclude the ``carr\'e du champ'' identity
$$ 2a_{XX} X f \otimes X f + 2a_{YY} Y f \otimes Y f + a_{XY} X f \otimes Y f + a_{XY} Y f \otimes X f = 0.$$
Since $X f$, $Y f$ are linearly independent, this implies that $a_{XX}, a_{YY}, a_{XY}$ vanish, which from \eqref{linear} implies that $X f, Y f, Z f$ are linearly dependent, which is absurd.  Thus the expression \eqref{xyzphi} is nowhere vanishing; as it descends to a continuous function on the compact space $H/\Gamma$, the claim \eqref{free-embed} follows.
\end{proof}

Using this proposition we can now dispose of the easy case when $\eps > 1/\log^{1/2} A$ of Theorem \ref{th3} (so that the orthogonality condition \eqref{ortho-1} does not need to be verified).  In this case we can set
$$ \phi_n(p) \coloneqq \iota( A^n \phi^0( \delta_{A^{-n}}(p) ) )$$
for all $n \in \Z$ and $p \in H$, where $\iota \colon \R^{20} \to \R^{29}$ is the standard embedding.  It is then a routine matter to use Proposition \ref{prop} and \eqref{chain} to verify all the conclusions of Theorem \ref{th3} (except for \eqref{ortho-1}, which does not need to be verified); note that the hypothesis $\eps > 1/\log^{1/2} A$ makes the contributions of the $n'>n$ terms in \eqref{nondeg} negligible.
Since $\eps$ avoids the interval $[1/A,1/\log^{1/2} A]$, we may thus assume henceforth that $\eps \leq 1/A$.

\begin{remark}  If one were to replace $\R^{29}$ in Theorem \ref{th3} by $\ell^2$, one could also easily conclude this variant of the theorem by setting
$$ \phi_n(p) \coloneqq \iota_n( A^n \phi^0( \delta_{A^{-n}}(p) ) )$$
for all $n \in \Z$ and $p \in H$, where $\iota_n \colon\R^{20} \to \ell^2$ are linear isometric embeddings of $\R^{20}$ into $\ell^2$ with pairwise orthogonal ranges.  This already recovers the Assouad embedding \cite{assouad} of $(H,d^{1-\eps})$ into $\ell^2$ with distortion $O( \eps^{-1/2})$.  We leave the details to the interested reader.
\end{remark}

We will shortly use Proposition \ref{prop} and Proposition \ref{induct} in an induction argument to establish the following technical claim.

\begin{claim}[Iteration]\label{consequence} Let $0 < \eps \leq 1/A$, and let $N_1 \leq N_2$ be integers.  Then one can find smooth functions $\phi_n \colon H \to \R^{29}$ for $N_1 \leq n \leq N_2$ obeying the following bounds, with $\phi_{\geq n} \colon H \to \R^{29}$ the function defined by the formula
$$ \phi_{\geq n} \coloneqq \sum_{n \leq n' \leq N_2} A^{-\eps(n'-n)} \phi_{n'}:$$
\begin{itemize}
\item (Smoothness at scale $A^n$) For all $N_1 \leq n \leq N_2$ and $p \in H$, one has
\begin{equation}\label{smooth-1}
\| \phi_n \|_{C^{20}_{A^n}} \leq C_0 A^n
\end{equation}
and
\begin{equation}\label{smooth-2}
\| \nabla^2 \phi_{\geq N_1} \|_{C^{18,\alpha}_{A^{N_1}}} \leq C_0 A^{-N_1}.
\end{equation}
\item (Orthogonality)  One has \eqref{ortho-1} for all $N_1 \leq n \leq N_2$.
\item (Non-degeneracy)  
For any $p \in H$ and $N_1 \leq n \leq N_2$, one has the estimates
\begin{align}
|X \phi_n(p)|, |Y \phi_n(p)| &\geq C_0^{-1} \label{iter-1} \\
|X \phi_{\geq N_1}(p) \wedge Y \phi_{\geq N_1}(p)| &\geq C_0^{-4} |X \phi_{\geq N_1}(p)| |Y \phi_{\geq N_1}(p)| \label{iter-2} \\
\left|\bigwedge_{W = X,Y,Z,XX,YY,XY} W \phi_{\geq n}(p)\right| &\geq 
C_0^{-17} A^{-4n} |X \phi_{\geq n}(p)| |Y \phi_{\geq n}(p)|.\label{iter-3}
\end{align}
\end{itemize}
\end{claim}

Suppose for the moment that we have Claim \ref{consequence}.  We now use this to show Theorem \ref{th3}.  We first observe that it suffices to construct, for each natural number $N$, a finite family $\phi_n = \phi_n^{(N)}$ for $-N \leq n \leq N$ of smooth maps from $H$ to $\R^{29}$ obeying the conclusions of Theorem \ref{th3} (with bounds independent of $N$) with the indices $n,n'$ restricted to $[-N,N]$ and with the $C^6_A$ norm replaced by $C^7_A$, since one can then apply the Arzel\'a-Ascoli theorem\footnote{Alternatively, one can take a limit as $N \to \infty$ along an ultrafilter.} and pass to a subsequence of $N$ for which the $\phi_n^{(N)}$ converge locally in the $C^6$ topology as $N \to \infty$ to limiting functions $\phi_n, n \in \Z$ that obey all the conclusions of Theorem \ref{th3} without any restriction on the parameter $n \in \Z$.  Next, for any given $N \geq 1$, we apply Claim \ref{consequence} with $N_1 = -N$ and $N_2 = N$ to obtain functions $\phi_n$, $-N \leq n \leq N$ obeying the properties \eqref{smooth-1}-\eqref{iter-3}.  The property \eqref{smooth-1} implies \eqref{smooth} for all $-N \leq n \leq N$; similarly, \eqref{iter-1} gives \eqref{nondeg-1}.  The property \eqref{ortho-1} for $-N \leq n \leq N$ is also true by construction.
The only estimate that requires some computation is \eqref{nondeg}.  But from \eqref{ortho-1}, the Pythagorean theorem, and induction we have
$$ |X \phi_{\geq n}(p)|^2 = \sum_{n \leq n' \leq N} A^{-2\eps(n'-n)} |X \phi_{n'}(p)|^2$$
and hence by \eqref{iter-1}, \eqref{smooth-1}
$$ C_0^{-2} M_{n,N}^2 \leq |X \phi_{\geq n}(p)|^2  \leq C_0^2 M_{n,N}^2$$
where 
\begin{equation}\label{mn-def}
M_{n,N}^2 \coloneqq \sum_{n \leq n' \leq N} A^{-2\eps(n'-n)} \sim \min(N-n, 1/\eps),
\end{equation}
and similarly for $Y \phi_{\geq n}(p)$.  Combining these bounds with \eqref{iter-3}, \eqref{smooth-1}, and Cauchy-Schwarz, we conclude that
$$ |X \phi_{\geq n}(p) \wedge Y \phi_{\geq n}(p) \wedge Z \phi_{\geq n}(p)| \gtrsim_{C_0} A^{-n} M_{n,N}^2,$$
which gives \eqref{nondeg} (with $n'$ restricted to $[-N,N]$) as required.  Thus Claim \ref{consequence} implies Theorem \ref{th3} and hence also Theorem \ref{th1}.

Now we derive Claim \ref{consequence} from Proposition \ref{induct}.  We do this by induction on the quantity $N_2-N_1$.   We first establish the base case when $N_2-N_1 = 0$.  By rescaling we may normalise $N_1=N_2=0$.  Let $\phi^0 \colon H \to \R^{20}$ be the map from Proposition \ref{prop}, then we simply set
$$ \phi_0 \coloneqq \iota \circ \phi^0$$
where $\iota \colon \R^{20} \to \R^{29}$ is the usual inclusion map.  All the properties of Claim \ref{consequence} are then immediate from Proposition \ref{prop} (for instance, the orthogonality \eqref{ortho-1} is trivial).

Now suppose that $N_2-N_1 > 0$, and the claim has already been proven for smaller values of $N_2-N_1$.  By rescaling we may assume $N_1=0 < N_2$.  Applying the inductive hypothesis with $N_1$ replaced by $1$, we can construct functions $\phi_n$ for all $1 \leq n \leq N_2$ obeying the conclusions of Claim \ref{consequence}.  In particular, if we write
$$ \psi \coloneqq A^{-\eps} \phi_{\geq 1} = \sum_{1 \leq n \leq N_2} A^{-\eps n} \phi_n$$
then (since $A^{-\eps} \sim 1$ when $\eps \leq 1/A$) we have the bounds
\begin{align}
\| \nabla^2 \psi \|_{C^{18,\alpha}_A} &\lesssim C_0 A^{-1} \label{smooth-2b} \\ 
|X \psi(p) \wedge Y \psi(p)| &\geq C_0^{-4} |X\psi(p)| |Y \psi(p)| \label{iter-2b} \\
\left|\bigwedge_{W = X,Y,Z,XX,YY,XY} W \psi(p)\right| &\gtrsim 
C_0^{-17} A^{-4} |X \psi(p)| |Y \psi(p)|\label{iter-3b}
\end{align}
for all $p \in H$, and our task is then to construct an additional function $\phi = \phi_0$ so that the bounds
\begin{align}
\| \phi \|_{C^{20}} &\leq C_0 \label{smooth-1a} \\ 
\| \nabla^2 (\phi + \psi) \|_{C^{18,\alpha}} &\leq C_0 \label{smooth-2a} \\ 
B( \phi, \psi) &= 0 \label{ortho-1a} \\
|X \phi(p)|, |Y \phi(p)| &\geq C_0^{-1} \label{iter-1a} \\
|X (\phi+\psi)(p) \wedge Y (\phi+\psi)(p)| &\geq C_0^{-4} |X(\phi+\psi)(p)| |Y (\phi+\psi)(p)| \label{iter-2a} \\
\left|\bigwedge_{W = X,Y,Z,XX,YY,XY} W (\phi+\psi)(p)\right| &\geq 
C_0^{-17} |X (\phi+\psi)(p)| |Y (\phi+\psi)(p)|\label{iter-3a}
\end{align}
hold for all $p \in H$.

From \eqref{ortho-1}, induction, and Pythagoras' theorem, we have for any $p \in H$ that
$$ |X\psi(p)|^2 = \sum_{1 \leq n \leq N_2} A^{-2\eps n} |X \phi_n(p)|^2$$
and hence by \eqref{iter-1}, \eqref{smooth-1} for $n_0=1$ we have
$$
 C_0^{-1} M \leq |X \psi(p)| \leq C_0 M
$$
where $M$ is the quantity
\begin{equation}\label{m-def}
 M \coloneqq \left(\sum_{1 \leq n \leq N_2} A^{-2\eps n}\right)^{1/2}.
\end{equation}
Similarly for $Y \psi(p)$, thus
\begin{equation}\label{xao}
 C_0^{-1} M \leq |X \psi(p)|, |Y \psi(p)| \leq C_0 M
\end{equation}  

We wish to invoke Proposition \ref{induct} for the indicated choices of $M, C_0$ to construct $\phi$.  To do this, we must first verify the hypotheses \eqref{xp-large}-\eqref{j-hold}  of that proposition.  The hypothesis \eqref{xp-large} is clear from \eqref{m-def} since $\eps \leq 1/A$, and the hypothesis \eqref{xyp-comp} follows from \eqref{xao}.  The hypothesis \eqref{xyp-nondeg} follows from \eqref{iter-2b}, and the hypothesis \eqref{free} similarly follows from \eqref{iter-3b}, \eqref{xao}.  Finally, \eqref{j-hold} follows from \eqref{smooth-2b}.  Thus we may apply Proposition \ref{induct} to locate a smooth map $\phi \colon  H \to \R^{29}$ with the stated properties \eqref{xyp}-\eqref{x-ortho}.

It remains to establish the required estimates \eqref{smooth-1a}-\eqref{iter-3a}.  The claim \eqref{smooth-1a} is immediate from \eqref{j-hold-phi}.  The latter estimate also gives
$$ \| \nabla^2 \phi \|_{C^{18,\alpha}} \lesssim 1$$
which when combined with \eqref{smooth-2b} gives \eqref{smooth-2a} (note that the factors of $A$ more than compensate for the additional factor of $C_0$).

The orthogonality property \eqref{ortho-1a} follows from \eqref{x-ortho}, and \eqref{iter-1a} follows from \eqref{xyp}, so we turn to \eqref{iter-2a}.  For brevity we omit dependence on $p$. Squaring both sides and using \eqref{x-ortho}, this claim is equivalent to
$$|X (\phi+\psi) \wedge Y (\phi+\psi)|^2 \geq C_0^{-8} (|X \phi|^2 + |X \psi|^2) (|Y \phi|^2 + |Y \psi|^2).$$
Comparing this with (the square of) \eqref{iter-2b}, we see that it suffices to show that
$$|X (\phi+\psi) \wedge Y (\phi+\psi)|^2 - |X \psi  \wedge Y \psi|^2 \geq C_0^{-8} (|X \phi|^2 |Y \psi|^2 + |X \psi|^2 |Y \phi|^2 + |X \phi|^2 |Y \phi|^2).$$
By \eqref{more-ortho}, the left-hand side is
$$ \gtrsim \frac{1}{C_0^{4}} M^2$$
while from \eqref{j-hold-phi} (and \eqref{xao}) the right-hand side is
$$ \lesssim C_0^{-8} C_0^2 M^2,$$
and \eqref{iter-2a} follows.  Finally, \eqref{iter-3a} follows from \eqref{free-inc}, \eqref{xao}, \eqref{j-hold-phi}.

To complete the proof of Theorem \ref{th1}, it thus remains to prove Proposition \ref{induct}.  This will be done in Section \ref{conclude}, after establishing a key perturbation theorem in Section \ref{perturb-sec} (which in turn relies on Littlewood-Paley theory for the Heisenberg group, which we review in Section \ref{lp-sec}), and some quantitative topological lemmas in Section \ref{top-sec}.

\section{Littlewood-Paley theory on the Heisenberg group}\label{lp-sec}

In order to construct a usable perturbation theory for the bilinear form $B$, we will need to introduce some basic Littlewood-Paley theory on the Heisenberg group.  This theory is developed in detail in \cite{bg1}, \cite{bg2} (see also \cite{thang}), but we will only need a more basic component of this theory from \cite{hulanicki}.  (See also the more general Littlewood-Paley theory on arbitrary manifolds developed in \cite{rodnianski}.)

Let $L$ denote the \emph{Laplacian-Kohn operator} (or \emph{sublaplacian})
$$ L \coloneqq - X^2 - Y^2.$$
This operator is self-adjoint on $L^2(H)$ (with the usual Haar measure $d\mu$ arising from Lebesgue measure on $\R^3$), and so by the bounded functional calculus one can define bounded operators $m(L)$ on $L^2(H)$ for any $m \in L^\infty(\R)$, which commute with each other and with $L$.  
In \cite{hulanicki} (see also \cite{bg1} for an alternate proof) it was shown that if $m \in C^\infty_c(\R)$, then this operator is given by convolution with a Schwartz function $K \colon H \to \R$ (using the usual definition of a Schwartz function arising from the identification of $H$ with $\R^3$), thus
\begin{equation}\label{mlf}
 m(L) f = K*f
\end{equation}
for any $f \in L^2(H)$, where the convolution operation $*$ is defined in the usual fashion as
$$ K*f(p) = \int_H K(g) f(g^{-1} p)\ d\mu(g) = \int_H K(p g^{-1}) f(g)\ d\mu(g).$$
In particular, for such $m$, the operator $m(L)$ can be extended to functions in $C^0$ using the formula \eqref{mlf}.

Let $\varphi \colon \R \to \R$ be a smooth function supported on $[-1,1]$ that equals $1$ on $[-1/2,1/2]$.  For any dyadic number $N$ (that is, a number of the form $2^n$ for an integer $n$), define the Littlewood-Paley projections $P_{\leq N}, P_{<N}, P_N, P_{\geq N}, P_{>N}$ using the aforementioned functional calculus by the formulae
\begin{align*}
P_{\leq N} &\coloneqq \varphi( L / N^2 ) \\
P_{<N} &= P_{\leq N/2} \\
P_N &\coloneqq P_{\leq N} - P_{<N} \\
P_{\geq N} &\coloneqq I - P_{<N} \\
P_{>N}  &\coloneqq I - P_{\leq N}
\end{align*}
where $I$ is the identity operator.  By the above discussion, each of these operators is well defined on $C^0$.

Intuitively, one should think of $P_{\leq N} \phi$ as a smooth restriction of $\phi$ to ``frequencies $\lesssim N$'', or to ``spatial scales $\gtrsim 1/N$''; similar interpretations exist for the other Littlewood-Paley operators.

We record some basic facts on how Littlewood-Paley projections interact with $C^k$ and $C^{k,\alpha}$ type spaces:

\begin{theorem}[Littlewood-Paley theory]\label{lp-theory}  Let $\phi \colon  H \to \R^D$ be bounded and smooth.
\begin{itemize}
\item[(i)]  (Scaling)  For any $\lambda>0$ and $N > 0$, we have
$$ P_{\leq N} ( \phi \circ \delta_\lambda ) = (P_{\leq N/\lambda} \phi ) \circ \delta_\lambda,$$
and similarly for $P_{<N}$, $P_N$, $P_{\geq N}$, $P_{>N}$.
\item[(ii)]  (Littlewood-Paley decomposition)  For any dyadic number $N_0$, we have
$$ \phi = P_{\leq N_0} \phi + \sum_{N>N_0} P_N \phi$$
where the sum is over dyadic numbers $N>N_0$, and the convergence is in the locally uniform topology.
\item[(iii)]  (Regularity)  If $N, M >0$ and $j,k \geq 0$, one has the estimates
\begin{align}
 \| \nabla^k P_{\leq N} \phi \|_{C^j_{1/N}} &\lesssim_{j,k} \| \nabla^k \phi \|_{C^0} \label{lp-2} \\
 \| P_{N} \phi \|_{C^{j}_{1/N}} &\lesssim_{j,k} N^{-k} \| \nabla^k \phi \|_{C^0}\label{lp-3} \\
 \| P_{N} \phi \|_{C^{j}_{1/N}} &\lesssim_{j,k} N^{-k-\alpha} \| \nabla^k \phi \|_{\dot C^{0,\alpha}}\label{lp-4}\\
 \| P_{N} \phi \|_{C^k_{1/M}}, \| P_{\leq N} \phi \|_{C^k_{1/M}}, \| P_{> N} \phi \|_{C^k_{1/M}} &\lesssim_{k} \| \phi \|_{C^k_{1/M}}. \label{lp-1}\\
 \| \nabla^j P_{> N} \phi \|_{C^0} &\lesssim_{j,k} N^{-k} \| \nabla^{j+k} \phi \|_{C^0}.\label{lp-5} 
\end{align}
\end{itemize}
\end{theorem}

Readers who are familiar with the Littlewood-Paley theory of Euclidean spaces should see that these estimates are perfectly analogous to their Euclidean counterparts.

\begin{proof}  From \eqref{chain} one has for any test function $\phi$ and $\lambda > 0$ that
$$ L (\phi \circ \delta_\lambda) = (\lambda^2 L \phi) \circ \delta_\lambda,$$
thus $\lambda^2 L$ is $L$ conjugated by the operation of composition with $\delta_\lambda$.  Since functional calculus respects conjugation, we conclude that
$$ m(L) (\phi \circ \delta_\lambda) = (m(\lambda^2 L) \phi) \circ \delta_\lambda$$
for any bounded multiplier $m \in L^\infty$; the claims in (i) then follow.

Now we prove (ii).  By telescoping series, it suffices to show that $P_{\leq N} \phi$ converges to $\phi$ as $N \to \infty$.  By the results of Hulanicki \cite{hulanicki}, the operator $P_{\leq 1}$ takes the convolution form
$$ P_{\leq 1} \phi(p) = \int_H K(g) \phi(g^{-1} p)\ d\mu(g)$$
for a Schwartz function $K$, hence by part (i)
\begin{equation}\label{pnp}
 P_{\leq N} \phi(p) = \int_H K(g) \phi((\delta_{1/N} g)^{-1} p)\ d\mu(g).
\end{equation}
To conclude it suffices to show that the Schwartz function $K$ has total mass $1$.  But by functional calculus, $P_{\leq N} \phi$ converges in $L^2$ to $\phi$ if $\phi \in L^2$, and this is only consistent with \eqref{pnp} if $K$ has total mass $1$.

Now we prove (iii).  By the scaling (i) we may take $N=1$.

We begin with \eqref{lp-2} for $N=1$.  It suffices to show that
$$
 \| \nabla^{j+k} P_{\leq 1} \phi \|_{C^0} \lesssim_{j,k} \| \nabla^k \phi \|_{C^0}.
$$
Write $P_{\leq 1} \phi = K * \phi$ for a Schwartz $K$, then $\nabla^{j+k} (K * \phi) = (\nabla^{j+k} K) * \phi$, and it suffices to show that
\begin{equation}\label{lp2-0}
 \| (\nabla^{j+k} K) * \phi \|_{C^0} \lesssim_{j,k,K} \| \nabla^k \phi \|_{C^0}
\end{equation}
for any Schwarz function $K$.  We prove this by induction on $k$.  The case $k=0$ is immediate from Young's inequality.  If $k \geq 1$ and the claim has already been proven for $k-1$, we write
 $$\| (\nabla^{j+k} K) * \phi \|_{C^0} \lesssim \| (\nabla^{j+k-1} X K) * \phi \|_{C^0} + \| (\nabla^{j+k-1} Y K) * \phi \|_{C^0}.$$
Integration by parts gives the identity
\begin{align*}
 (XK) * \phi &= ((\tilde X + yZ) K) * \phi \\
&= (\tilde X K) * \phi + Z(yK) * \phi \\
&= K * X \phi + yK * Z \phi \\
&= K * X \phi + yK * (YX-XY) \phi \\
&= K * X \phi + (\tilde Y(yK)) * X \phi - (\tilde X(yK)) * Y \phi
\end{align*}
where $\tilde X \coloneqq \frac{\partial}{\partial x}$, $\tilde Y \coloneqq \frac{\partial}{\partial y} + x \frac{\partial}{\partial z}$ are the left-invariant counterparts to the right-invariant vector fields $X,Y$, and (by slight abuse of notation) $y$ denotes the coordinate function $[x,y,z] \mapsto y$ on $H$.  Thus one has a representation
$$ (XK) * \phi = K' * (X\phi) + K'' * Y \phi$$
for some further Schwartz functions $K', K''$, which implies
$$ (\nabla^{j+k-1} X K) * \phi = (\nabla^{j+k-1} K') * X \phi + (\nabla^{j+k-1} K'') * Y \phi$$
and hence by the induction hypothesis
\begin{align*}
\| (\nabla^{j+k-1} X K) * \phi \|_{C^0}  &\lesssim_{j,k,K} \| \nabla^{k-1} X \phi \|_{C^0} + \| \nabla^{k-1} Y \phi \|_{C^0}  \\
&\lesssim \| \nabla^k \phi \|_{C^0}.
\end{align*}
Similarly for $(\nabla^{j+k-1} Y K) * \phi$.  This establishes \eqref{lp2-0} and hence \eqref{lp-2}.

Now we turn to the $N=1$ case of \eqref{lp-3}.  It suffices to show that
$$ \| \nabla^{l} P_1 \phi \|_{C^0} \lesssim_{l,k} \| \nabla^k \phi \|_{C^0}$$
for all $l \geq 0$. We can factor $P_1 = m_k(L) L^k$ for some $m_k \in C^\infty_c(\R)$, and hence
$$ \nabla^l P_1 \phi = \nabla^l K_{k} * L^k \phi$$
where $K_k$ is the convolution kernel of $m_k(L)$.  The operator $L^k$ is of order $2k$; using integration by parts to move $k$ of these derivatives onto $\nabla^l K_{k}$, we conclude a representation formula
$$ \nabla^l P_1 \phi = \sum_D K_{k,l,D} * D \phi$$
where $D$ ranges over the components of $\nabla^k$ and the $K_{k,l,D}$ are Schwartz functions.  The claim then follows from Young's inequality.

To adapt the above argument to prove \eqref{lp-4}, we would have to establish the estimate
$$ \| \nabla^l (L^k)^* K_k * \phi \|_{C^0} \lesssim_{l,k} \| \phi\|_{\dot C^{0,\alpha}}.$$
Since $m_k$ vanishes near the origin, $m_k(L)$ annihilates $1$, and hence $K_k$ has mean zero; thus $\nabla^l (L^k)^* K_k$ also has mean zero. We can then write
$$ \nabla^l (L^k)^* K_k * \phi(p) = \int_H \nabla^l (L^k)^* K_k(g) (\phi(g^{-1} p) - \phi(p))\ d\mu(g);$$
bounding $\phi(g^{-1} p) - \phi(p) = O( \| \phi \|_{\dot C^{0,\alpha}} d(0,g)^{\alpha})$ and using the Schwartz nature of $K_k$, we obtain the claim.

Now we establish the $N=1$ case of \eqref{lp-1}.  It suffices by the triangle inequality to establish the claim for $P_{\leq 1}$.
It suffices to show that
$$ \| \nabla^j P_{\leq 1} \phi \|_{C^0} \lesssim_j \| \nabla^j \phi \|_{C^0}$$
for all $j \geq 0$.
But this follows from \eqref{lp-2}.

Finally, \eqref{lp-5} follows from \eqref{lp-2} and the triangle inequality when $k=0$, and from \eqref{lp-3} and the triangle inequality when $k \geq 1$.
\end{proof}

\section{Perturbation theory for a bilinear form}\label{perturb-sec}

As mentioned in the introduction, a key aspect of Proposition \ref{induct} is finding, for a given $\psi \colon H \to \R^{29}$, a ``good'' solution $\phi \colon  H \to \R^{29}$ to the differential equation \eqref{x-ortho} which is as smooth as $\psi$.  To solve this equation, we will first develop a perturbative theory in which we find a solution $\phi$ to the equation
\begin{equation}\label{bf}
 B(\phi,\psi) = F
\end{equation}
for given $\psi$, $F$, with bounds on $\phi$ in terms of $\psi$ and $F$.  Note that this system is highly underdetermined (two equations in $29$ unknowns), so solutions will be far from unique.

The perturbation theorem we will prove (using a variant of the Nash-Moser iteration scheme) will start with an approximate solution $\tilde \phi$ solving a low frequency analogue
\begin{equation}\label{lfe}
 B(\tilde \phi, P_{\leq N_0} \psi) = 0
\end{equation}
of \eqref{bf}, and find a solution $\phi$ to \eqref{bf} that is close to \eqref{bf} in a good norm (specifically, $C^{20,\alpha}$ norm) if $F$ is suitably small and $\tilde \phi$ is not too large.  In Section \ref{conclude} we will apply this theorem with a suitable explicit choice of $\tilde \phi$.

To solve the equation \eqref{bf} with $F = (F_X,F_Y)$, it would suffice by \eqref{b-expand} to solve the linear system of equations
\begin{align*}
\phi \cdot X \psi &= 0 \\
\phi \cdot XX \psi &= - F_X \\
\phi \cdot Y \psi &= 0 \\
\phi \cdot YY \psi &= - F_Y.
\end{align*}
As there are no derivatives being placed on $\phi$, this task is easily accomplished when one has the freeness hypothesis that $X\psi(p), XX \psi(p), Y \psi(p), YY \psi(p)$ are linearly independent at each $p$.  Indeed, if for each $p$ we let $T_\psi(p) \colon \R^{29} \to \R^4$ denote the linear map
$$ T_\psi(p) v \coloneqq ( v \cdot X \psi(p), v \cdot XX \psi(p), v \cdot Y \psi(p), v \cdot YY \psi(p) )$$
then the freeness hypothesis asserts that $T_\psi(p)$ has full rank (in a certain quantitative sense), and if one defines the \emph{pseudoinverse} $T_\psi(p)^{-1} \colon \R^4 \to \R^{29}$ of $T_\psi(p)$ by the formula
$$ T_\psi(p)^{-1} \coloneqq T_\psi(p)^* (T_\psi(p) T_\psi(p)^*)^{-1} $$
where $T_\psi(p)^* \colon \R^4 \to \R^{29}$ is the adjoint map to $T_\psi(p)$, then $T_\psi(p) T_\psi(p)^{-1}$ is the identity on $\R^4$, thus we have the pointwise identities
\begin{equation}\label{abcd}
\begin{split}
T_\psi^{-1}(a,b,c,d) \cdot X \psi &= a \\
T_\psi^{-1}(a,b,c,d) \cdot XX \psi &= b \\
T_\psi^{-1}(a,b,c,d) \cdot Y \psi &= c \\
T_\psi^{-1}(a,b,c,d) \cdot YY \psi &= d
\end{split}
\end{equation}
for any smooth $a,b,c,d \colon H \to \R$, and one has the explicit solution
\begin{equation}\label{phi}
 \phi(p) \coloneqq T_\psi(p)^{-1}( 0, -F_X(p), 0, -F_Y(p) )
\end{equation}
to \eqref{bf}.

Unfortunately, this solution to \eqref{bf} has a significant drawback for our purposes: the presence of derivatives $X \psi, XX \psi, Y \psi, YY \psi$ in the definition of $T_\psi$ will ensure that the solution $\phi$ is less regular than the function $\psi$, which is unacceptable for the purposes of proving Proposition \ref{induct}, due to our need to iteratively apply this theorem in the proof of Claim \ref{consequence}.  In particular, the $C^{20,\alpha}$ type control on $\psi$ only gives $C^{18,\alpha}$ type control on $T_\psi$, and hence on $\phi$.  To not lose derivatives, and recover a solution $\phi$ in the high regularity space $C^{20,\alpha}$, we will need a more complicated solution to \eqref{bf} than \eqref{phi} constructed by a Nash-Moser type scheme.  More precisely, we show the following perturbation theorem.

\begin{proposition}[Perturbation theorem]\label{perturb}  Let $M$, $\psi$ be as in Proposition \ref{induct}, and let $F\colon H \to \R^2$ be a smooth function such that $\|F\|_{C^{24}} < \infty$.  Let $\tilde \phi \colon  H \to \R^{29}$ be a smooth solution to the low frequency equation \eqref{lfe} with $\| \tilde \phi \|_{C^{40}} < \infty$.  Then there exists a smooth solution $\phi$ to \eqref{bf} obeying the bound
\begin{equation}\label{ax}
 \| \phi - \tilde \phi \|_{C^{20,\alpha}} \lesssim_{C_0} \eps
\end{equation}
and also the variant estimate
\begin{equation}\label{ax-2}
 \| X(\phi-\tilde \phi) \cdot Y \psi \|_{C^0}, \| Y(\phi-\tilde \phi) \cdot X \psi \|_{C^0} \lesssim_{C_0} A^{-1} \eps
\end{equation}
where $\eps$ is the quantity
\begin{equation}\label{eps-def}
 \eps \coloneqq A \|F\|_{C^{24}} + A^{-10} \| \tilde \phi \|_{C^{40}}.
\end{equation}
\end{proposition}

Note here that there is some loss of derivatives when passing from $\tilde \phi, F$ to $\phi$, as $C^{20,\alpha}$ is less regular than $C^{24}$ or $C^{40}$; however, we will only apply this proposition with the approximate solution $\tilde \phi$ and the error term $F$ being of low-frequency  or of ``high-high paraproduct'' type respectively, and as such will lie in regular spaces such as $C^{40}$ or $C^{24}$ with even some room to spare. In fact we will ultimately take $F=0$, though for iteration purposes it is convenient to state the proposition in a manner that allows for non-zero $F$.   The technical variant \eqref{ax-2} of \eqref{ax} is needed to ensure that certain cross-terms arising in the increment property \eqref{more-ortho} are of manageable size (and in particular do not generate an unwanted factor of $M$ in the estimates, which would otherwise arise if one naively estimated these dot products using \eqref{xyp-comp} and the Cauchy-Schwarz inequality).

\begin{proof}  It will suffice to find a $\phi$ with the stated bounds solving the approximate equation
\begin{equation}\label{phi-bounds}
\| B(\phi,\psi) - F \|_{C^{24}} \lesssim_{C_0} A^{-2} \eps,
\end{equation}
rather than \eqref{bf}, since one can then iteratively replace $(\tilde \phi, F)$ by the residual $(0, F - B(\phi,\psi))$ (which reduces $\eps$ to $O_{C_0}(A^{-1} \eps)$) and sum the resulting Neumann series to obtain an exact solution to \eqref{bf}, thanks to the linearity of this equation in $\phi$ and $F$.  (Strictly speaking, one should not sum the infinite Neumann series, as the resulting sum $\phi$ will then only lie in $C^{20,\alpha}$ rather than being smooth; instead, one should sum the Neumann series out to some extremely large finite length so that the $C^{24}$ norm of the residual $F$ becomes extremely small, and then apply the explicit solution \eqref{phi} to eliminate this last residual, which will be acceptable if the $C^{24}$ norm of $F$ is small enough since $\psi$ is assumed to be smooth.)

In the spirit of the Nash-Moser iteration scheme, we construct the (approximate) solution $\phi$ in stages, starting with a ``low frequency'' component $\phi_{\leq N_0}$ that solves a low-frequency equation
$$ B( \phi_{\leq N_0}, P_{\leq N_0} \psi ) = P_{\leq N_0} F $$
and then iteratively adding on higher frequency components $\phi_N$, $N > N_0$ in order to approximately solve a higher-frequency equation
\begin{equation}\label{pnf}
 B( \phi_{\leq N}, P_{\leq N} \psi ) \approx P_{\leq N} F.
\end{equation}
As in the Nash-Moser scheme, we will need to apply a mollification to $\phi$ at each stage in order to counteract the loss of derivatives problem; this explains the presence of the Littlewood-Paley projection $P_{\leq N}$ applied to the $\phi_{<N}$ type terms in the construction that follows.  (This also explains the need to allow an inhomogeneous term $F$ in \eqref{bf}, even though in our applications we will eventually set this term equal to zero.)

We turn to the details.   
Write $F = (F_X,F_Y)$, and define the function
\begin{equation}\label{phiny}
 \phi_{\leq N_0} \coloneqq \tilde \phi + T_{P_{\leq N_0} \psi}^{-1}( 0, -P_{\leq N_0} F_X, 0, -P_{\leq N_0} F_Y );
\end{equation}
observe from \eqref{phiny}, \eqref{b-expand}, \eqref{abcd}, \eqref{lfe} that one has
\begin{equation}\label{b-ident-1}
 B( \phi_{\leq N_0}, P_{\leq N_0} \psi ) = P_{\leq N_0} F.
\end{equation}
Next, for every dyadic $N > N_0$ we recursively define $\phi_N$ by the formula
\begin{equation}\label{phin-form}
 \phi_N \coloneqq T_{P_{\leq N} \psi}^{-1}( a_N, b_N, c_N, d_N )
\end{equation}
where
\begin{align*}
a_N &\coloneqq - (X P_{\leq N} \phi_{<N}) \cdot P_N \psi \\
b_N &\coloneqq - (XX P_{\leq N} \phi_{<N}) \cdot P_N \psi - P_N F_X \\
c_N &\coloneqq - (Y P_{\leq N} \phi_{<N}) \cdot P_N \psi \\
d_N &\coloneqq - (YY P_{\leq N} \phi_{<N}) \cdot P_N \psi - P_N F_Y
\end{align*}
and
$$ \phi_{<N} \coloneqq \phi_{\leq N_0} + \sum_{N_0 < M < N} \phi_M$$
and similarly
$$ \phi_{\leq N} \coloneqq \phi_{\leq N_0} + \sum_{N_0 < M \leq N} \phi_M$$
and $M$ is understood to range over dyadic numbers.  Observe from \eqref{b-ident}, \eqref{abcd} that
\begin{equation}\label{b-ident-2}
 B( \phi_N, P_{\leq N} \psi ) + B( P_{\leq N} \phi_{<N}, P_N \psi ) = P_N F
\end{equation}
(compare with \eqref{b-diff}).  As mentioned in the introduction, the rather complicated-looking form of $\phi_N$ (which the author only arrived at after some trial and error) is chosen so that no derivatives are placed on $P_N \psi$, and also there is some mollification of the $\phi_{<N}$ term in order to avoid the loss of derivatives problem.

We have the following estimates on $T_{P_{\leq N} \psi}^{-1}$:

\begin{lemma}[Controlling the pseudoinverse]\label{tpsi}  For any $N \geq N_0$, one has
$$ 
\| T_{P_{\leq N} \psi}^{-1} \|_{C^{18}_A} \lesssim_{C_0} A
$$
and
$$
\| \nabla^{18} T_{P_{\leq N} \psi}^{-1} \|_{C^{50}_{1/N}} \lesssim_{C_0} A^{-17}.$$
\end{lemma}

\begin{proof}  We abbreviate $T_{P_{\leq N} \psi}$ as $T$, then we can write the pseudoinverse $T_{P_{\leq N} \psi}^{-1} = T^{-1}$ as
$$ T^{-1} = \frac{1}{\det(TT^*)} T^* \mathrm{adj}(T T^*)$$
where $\mathrm{adj}(A)$ denotes the adjugate matrix of $A$.
Our task is then to show the bounds
$$ \left\| \nabla^j\left( \frac{1}{\det(TT^*)} T^* \mathrm{adj}(T T^*) \right) \right\|_{C^0} \lesssim_{C_0} A B_j$$
for $0 \leq j \leq 68$, where
$$ B_j \coloneqq A^{-j} + N^{j-18} A^{-18}.$$
The main difficulty here is not to lose a factor of $M$, which would be quite problematic for other parts of the argument.  (Actually, when computed carefully, some terms even gain a factor of $M$, but we will not exploit this.)

From \eqref{j-hold}, \eqref{xyp-comp}, and Theorem \ref{lp-theory}(iii), we have
$$ |\nabla^j W P_{\leq N} \psi| \lesssim_{C_0} M B_j$$
when $0 \leq j \leq 68$ and $W \in \{X,Y\}$ (indeed when $j>0$ we can delete the factor of $M$), and
$$ |\nabla^j W P_{\leq N} \psi| \lesssim_{C_0} A^{-1} B_j$$
when $0 \leq j \leq 68$ and $W \in \{XX, YY \}$.  Thus, viewing $T$ as a $4 \times 29$ matrix, with the second and third rows permuted to place the rows associated to first-order operators $X,Y$ on top and to second-order operators $XX, YY$ on the bottom, the first two rows of $\nabla^j T$ have norm $O_{C_0}(M B_j)$, and the bottom two have norm $O_{C_0}(A^{-1} B_j)$.  By the product rule (and noting that $B_j B_k \lesssim B_{j+k}$ for all $j,k \geq 0$) we conclude that the $4 \times 4$ matrix $\nabla^j(TT^*)$ has top left $2 \times 2$ block of size $O_{C_0}( M^2 B_j )$, the top right and bottom left blocks have size $O_{C_0}(A^{-1} M B_j)$, and the bottom right has size $O_{C_0}(A^{-2} B_j)$.  By the product rule and cofactor expansion, $\nabla^j \mathrm{adj}(TT^*)$ then has top left block of size $O_{C_0}(M^2A^{-4} B_j)$, top right and bottom left blocks of size $O_{C_0}( M^3 A^{-3} B_j)$, and bottom right block of size $O_{C_0}( M^4 A^{-2} B_j)$.  By the product rule, the $29 \times 4$ matrix $\nabla^j( T^* \mathrm{adj}(TT^*) )$ then has all rows of size $O_{C_0}( M^4 A^{-3} B_j )$ (some are lower order than this).  

Similarly, $\nabla^j( \det(TT^*) )$ has magnitude $O_{C_0}(M^4 A^{-4} B_j)$.  
Meanwhile, from \eqref{free}, \eqref{xyp-comp}, \eqref{j-hold}, and using \eqref{lp-5} to approximate $P_{\leq N} \psi$ by $\psi$ up to negligible error, we see that the vectors $W P_{\leq N} \psi$ for $W \in \{X,Y,Z,XX,YY,XY\}$ have magnitude $O_{C_0}(M)$ when $W \in \{X,Y\}$ and $O_{C_0}(A^{-1})$ for $W \in \{Z,XX,YY,XY\}$, with wedge product lower bound
$$ \left|\bigwedge_{W = X,Y,Z,XX,YY,XY} W P_{\leq N} \psi\right| \gtrsim_{C_0} M^2 A^{-4}.$$
In particular by Cauchy-Schwarz we have
$$ \left|\bigwedge_{W = X,Y,XX,YY} W P_{\leq N} \psi\right| \gtrsim_{C_0} M^2 A^{-2}.$$
From this and the Cauchy-Binet formula \eqref{cauchy-binet} we have the matching lower bound
$$ \det(TT^*) \gtrsim_{C_0} M^4 A^{-4}$$
for the determinant.  Hence by the quotient rule, $\nabla^j( \det(TT^*)^{-1} )$ has magnitude $O_{C_0}( M^{-4} A^4 B_j)$.  The claim now follows from the product rule.
\end{proof}

From this proposition (and the fact that $A$ is large compared with $N_0$) we have the estimate
$$
\| T_{P_{\leq N_0} \psi}^{-1} \|_{C^{40}} \lesssim_{C_0} A 
$$
while from Theorem \ref{lp-theory}(iii) we have
$$ \| P_{\leq N_0} F \|_{C^{24}} \lesssim \|F\|_{C^{24}}$$
and
$$ \| P_{\leq N_0} F \|_{C^{40}} \lesssim_{N_0} \|F\|_{C^{24}}$$
and thus from \eqref{phiny}, \eqref{algebra-1}, \eqref{eps-def}
\begin{equation}\label{sim}
 \| \phi_{\leq N_0} - \tilde \phi \|_{C^{24}} \lesssim_{C_0} \eps
\end{equation}
as well as the variant
\begin{equation}\label{sim-2}
 \| \phi_{\leq N_0} - \tilde \phi \|_{C^{40}} \lesssim_{C_0,N_0} \eps.
\end{equation}
From \eqref{sim-2}, \eqref{eps-def}, and the triangle inequality, we also have
\begin{equation}\label{sim-2a}
 \| \phi_{\leq N_0} \|_{C^{40}} \lesssim_{C_0,N_0} A^{10} \eps.
\end{equation}

Next, from Theorem \ref{lp-theory}(iii) we have
\begin{align*}
 \| \nabla^k P_{\leq N} \phi_{<N} \|_{C^{40}_{1/N}} &\lesssim \| \nabla^k \phi_{<N} \|_{C^0} \\
& \lesssim \| \phi_{<N} \|_{C^2}
\end{align*}
for $k=0,1,2$.  This implies in particular that
$$ \| W P_{\leq N} \phi_{<N} \|_{C^{40}_{1/N}}  \lesssim \| \phi_{<N} \|_{C^2}$$
for $W = X, Y, XX, YY$.

Further application of Theorem \ref{lp-theory}(iii) also yields the estimates
\begin{align*}
\| P_N \psi \|_{C^{40}_{1/N}} &\lesssim N^{-20-\alpha} \| \nabla^{20} \psi \|_{\dot C^{0,\alpha}} \\
&\lesssim  N^{-20-\alpha} A^{-18} \| \nabla^2 \psi \|_{C^{18,\alpha}_A} \\
&\lesssim_{C_0} N^{-20-\alpha} A^{-19} 
\end{align*}
and
\begin{align*} 
\| P_N F \|_{C^{40}_{1/N}} &\lesssim N^{-24} \| \nabla^{24} F \|_{C^0} \\
&\lesssim N^{-24} \| F \|_{C^{24}} \\
&\lesssim N^{-24} A^{-1} \eps.
\end{align*}

Finally, from Lemma \ref{tpsi} one has
$$ \| T_{P_{\leq N} \psi}^{-1} \|_{C^{40}_{1/N}} \lesssim_{C_0} A.$$ 
since $A^{1-j} \lesssim A N^j$ for $0 \leq j \leq 18$ and $A^{-17} N^{j-18} \lesssim A N^j$ for $18 \leq j \leq 40$.  

Inserting the above estimates and \eqref{algebra-1} into \eqref{phin-form}, we conclude that
\begin{equation}\label{oick}
 \| \phi_N \|_{C^{40}_{1/N}} \lesssim_{C_0} ( N^{-24} \eps + N^{-20-\alpha} A^{-18} \| \phi_{<N} \|_{C^2} ) 
\end{equation}
and in particular
$$ \| \phi_N \|_{C^2} \lesssim_{C_0} N^{-22} \eps + N^{-18-\alpha} A^{-18} \| \phi_{<N} \|_{C^2}.$$
By the triangle inequality and \eqref{eps-def} we thus have
\begin{align*}
 \| \phi_{\leq N} - \tilde \phi \|_{C^2} &\leq  \| \phi_{< N} - \tilde \phi \|_{C^2} + \| \phi_N \|_{C^2} \\ 
&\leq (1 + O_{C_0}(A^{-18} N^{-18-\alpha})) \| \phi_{<N} - \tilde \phi \|_{C^2} + O_{C_0}( A^{-18} N^{-18-\alpha} ) \| \tilde \phi \|_{C^2} + O_{C_0}(N^{-22} \eps) \\
&\leq (1 + O_{C_0}(A^{-18} N^{-18-\alpha})) \| \phi_{<N} - \tilde \phi \|_{C^2} + O_{C_0}( A^{-8} N^{-18-\alpha} \eps ) + O_{C_0}(N^{-22} \eps).
\end{align*}
Iterating this (using the discrete form of Gronwall's inequality) starting with \eqref{sim}, we conclude that
$$ \| \phi_{\leq N} - \tilde \phi \|_{C^2} \lesssim_{C_0} \eps $$
for any $N \geq N_0$, which by the triangle inequality and \eqref{eps-def} implies that
$$ \| \phi_{\leq N} \|_{C^2} \lesssim_{C_0} A^{10} \eps.$$
Inserting this back into \eqref{oick} we conclude that
\begin{equation}\label{sim-3}
 \| \phi_N \|_{C^{40}_{1/N}} \lesssim_{C_0} N^{-24} \eps + A^{-8} N^{-20-\alpha} \eps
\end{equation}
which implies in particular that the sum
$$ \phi \coloneqq \phi_{\leq N_0} + \sum_{N > N_0} \phi_N$$
converges in (say) the $C^2$ topology.

We now prove \eqref{ax}.  From \eqref{sim} and the triangle inequality it suffices to show that
$$
\left\| \sum_{N > N_0} \phi_N \right\|_{C^{20,\alpha}} \lesssim_{C_0} \eps.$$
From \eqref{sim-3} we have
$$
\| \phi_N \|_{C^{20}} \lesssim_{C_0} N^{-4} \eps + A^{-8} N^{-\alpha} \eps$$
and hence by the triangle inequality
$$
\| \sum_{N > N_0} \phi_N \|_{C^{20}} \lesssim_{C_0} \eps$$
(with some room to spare).  Thus it will suffice to show that
\begin{equation}\label{hold}
|\nabla^{20} \sum_{N > N_0} \phi_N(p) - \nabla^{20} \sum_{N > N_0} \phi_N(q)| \lesssim_{C_0} \eps d(p,q)^\alpha
\end{equation}
for any $p,q \in H$.  By the triangle inequality, the left-hand side of \eqref{hold} is at most
$$ \sum_{N>N_0} |\nabla^{20} \phi_N(p) - \nabla^{20} \phi_N(q)|.$$
On one hand, we may bound 
\begin{align*}
|\nabla^{20} \phi_N(p) - \nabla^{20} \phi_N(q)| &\lesssim \| \nabla^{20} \phi_N \|_{C^0} \\
&\lesssim N^{20} \| \phi_N \|_{C^{40}_{1/N}} \\
&\lesssim_{C_0} N^{-4} \eps + A^{-8} N^{-\alpha} \eps \\
&\lesssim N^{-\alpha} \eps.
\end{align*}
On the other hand, one has
\begin{align*}
|\nabla^{20} \phi_N(p) - \nabla^{20} \phi_N(q)| &\lesssim \| \nabla^{21} \phi_N \|_{C^0} d(p,q)\\
&\lesssim N^{21} \| \phi_N \|_{C^{40}_{1/N}} d(p,q)\\
&\lesssim_{C_0} (N^{-3} \eps + A^{-8} N^{1-\alpha} \eps) d(p,q)\\
&\lesssim N^{-\alpha} \eps (N d(p,q)).
\end{align*}
Thus the left-hand side of \eqref{hold} is bounded by
$$ \lesssim_{C_0} \eps \sum_N N^{-\alpha} \min( 1, N d(p,q))$$
and the claim \eqref{ax} follows by summing geometric series and using the hypothesis $0 < \alpha < 1$.

For future reference we observe that the above argument also gives the bound
\begin{equation}\label{ax-alt}
 \| \phi_{\leq N} - \tilde \phi \|_{C^{20,\alpha}} \lesssim_{C_0} \eps
\end{equation}
for any $N\geq N_0$.

Now we prove \eqref{phi-bounds}.  As $\phi_{\leq N}$ converges in $C^2$ to $\phi$ as $N \to \infty$, and $P_{\leq N} \psi$ converges in $C^2$ to $\psi$, we may write $B(\phi,\psi)$ as the uniform limit of $B(\phi_{\leq N}, P_{\leq N} \psi)$.  This telescopes to
\begin{align*}
B(\phi,\psi) &= B(\phi_{\leq N_0}, P_{\leq N_0} \psi) + \sum_{N > N_0} (B(\phi_{\leq N}, P_{\leq N} \psi) - B(\phi_{< N}, P_{< N} \psi)) \\
&=  B(\phi_{\leq N_0}, P_{\leq N_0} \psi) + \sum_{N > N_0} B(\phi_N, P_{\leq N} \psi) + B( P_N \psi, \phi_{<N} )
\end{align*}
where we have used the symmetry of $B$.  
From this and \eqref{b-ident-1}, \eqref{b-ident-2} we conclude that
$$ B(\phi,\psi) - F = \sum_{N>N_0} B( P_N \psi, P_{> N} \phi_{<N}).$$
The right-hand side is a ``high-high paraproduct'' of $\nabla \psi$ and $\nabla \phi$, and as such will have significantly more regularity than either $\nabla \psi$ or $\nabla \phi$ separately (closer to $C^{40}$ type regularity than $C^{20}$ type).  Indeed, by the triangle inequality we have
$$ \| B(\phi,\psi) - F\|_{C^{24}} \leq \sum_{N>N_0} \| B( P_N \psi, P_{> N} \phi_{<N}) \|_{C^{24}}.$$
Using the original form \eqref{bilinear-def} of $B$ and the product rule, the right-hand side is bounded by
$$ \lesssim \sum_{N>N_0} \sum_{j_1+j_2 = 24} \| \nabla P_N \psi \|_{C^{j_1}} \| \nabla P_{> N} \phi_{<N} \|_{C^{j_2}}.$$
Using Theorem \ref{lp-theory}(iii), \eqref{j-hold}, we have for any $0 \leq j_1 \leq 24$ that
\begin{align*}
\| \nabla P_N \psi \|_{C^{j_1}} &\lesssim N^{j_1+1} \| P_N \psi \|_{C^{j_1+1}_{1/N}} \\
&\lesssim N^{j_1+1} N^{-20-\alpha} \| \nabla^{20} \psi \|_{\dot C^{0,\alpha}} \\
&\lesssim N^{j_1-19-\alpha} A^{-18-\alpha} \| \nabla^2 \psi \|_{C^{18,\alpha}_A} \\
&\lesssim_{C_0} N^{j_1-19-\alpha} A^{-19-\alpha}.
\end{align*}
In a similar spirit, for any $0 \leq j_2 \leq 24$ one has from Theorem \ref{lp-theory}(iii), \eqref{sim-2a}, \eqref{sim-3}, (and being somewhat inefficient with the estimates) that
\begin{align*}
\| \nabla P_{> N} \phi_{<N} \|_{C^{j_2}} &\lesssim N^{j_2} \| \nabla P_{>N} \phi_{<N} \|_{C^{j_2}_{1/N}} \\
&\lesssim N^{j_2 - 39} \| \nabla^{40} \phi_{<N} \|_{C^0} \\
&\lesssim N^{j_2 - 39} ( \| \nabla^{40} \phi_{\leq N_0} \|_{C^0} + \sum_{N_0 < N' < N} \| \nabla^{40} \phi_{N'} \|_{C^0} ) \\
&\lesssim_{C_0,N_0} N^{j_2 - 39} \left( A^{10} \eps + \sum_{N_0 < N' < N} (N')^{40} ( (N')^{-24} \eps + A^{-8} (N')^{-20-\alpha} \eps ) \right) \\
&\lesssim_{C_0,N_0} N^{j_2 - 19} A^{10} \eps    
\end{align*}
and thus
$$ \| B(\phi,\psi) - F\|_{C^{24}} \lesssim_{C_0,N_0} N^{24 - 19 - 19 - \alpha} A^{10-19 - \alpha} \eps$$
which gives \eqref{phi-bounds} with some room to spare.

Finally, we prove \eqref{ax-2}.  We just establish the estimate for $X(\phi-\tilde \phi) \cdot Y \psi$, as the estimate for $Y(\phi-\tilde \phi) \cdot X \psi$ is completely analogous.  By the Leibniz rule we have
$$ X(\phi-\tilde \phi) \cdot Y \psi = X ( (\phi-\tilde \phi) \cdot Y \psi ) - (\phi-\tilde \phi) \cdot XY \psi $$
and hence by the triangle inequality we have
$$ \| X(\phi-\tilde \phi) \cdot Y \psi\|_{C_0} \leq \| (\phi-\tilde \phi) \cdot Y \psi \|_{C^1} + \| \phi - \tilde \phi \|_{C^0} \| XY \psi \|_{C^0}.$$
The second term is acceptable thanks to \eqref{ax}, \eqref{j-hold}, so it remains to show that
$$ \| (\phi-\tilde \phi) \cdot Y \psi \|_{C^1} \lesssim_{C_0} A^{-1} \eps.$$
By the triangle inequality, the left-hand side is at most
$$ \| (\phi_{\leq N_0} -\tilde \phi) \cdot Y P_{\leq N_0} \psi \|_{C^1} + 
\| (\phi_{\leq N_0} -\tilde \phi) \cdot Y P_{> N_0} \psi \|_{C^1} + \sum_{N > N_0}
\| \phi_N \cdot Y P_{\leq N} \psi \|_{C^1} + 
\| \phi_N \cdot Y P_{> N} \psi \|_{C^1}.$$
From \eqref{phiny}, \eqref{abcd} one has
$$ (\phi_{\leq N_0} -\tilde \phi) \cdot Y P_{\leq N_0} \psi  = 0.$$
From \eqref{algebra-1}, \eqref{sim}, Theorem \ref{lp-theory}(iii), \eqref{j-hold} one has
\begin{align*}
\| (\phi_{\leq N_0} -\tilde \phi) \cdot Y P_{> N_0} \psi \|_{C^1} &\lesssim
\|\phi_{\leq N_0} -\tilde \phi\|_{C^1} \| Y P_{> N_0} \psi \|_{C^1} \\
&\lesssim_{C_0} \eps N_0^2 \| P_{>N_0} \psi \|_{C^2_{1/N_0}} \\
&\lesssim_{C_0} \eps \| \nabla^2 \psi \|_{C^0} \\
&\lesssim_{C_0} \eps A^{-1}.
\end{align*}
From \eqref{phin-form}, \eqref{abcd} one has
$$ \phi_N \cdot Y P_{\leq N} \psi = - (Y P_{\leq N} \phi_{<N}) \cdot P_N \psi$$
and hence by \eqref{algebra-1}, Theorem \ref{lp-theory}(iii), \eqref{ax-alt}, \eqref{eps-def}, \eqref{j-hold}, one has
\begin{align*}
\|\phi_N \cdot Y P_{\leq N} \psi \|_{C^1} &\lesssim \| Y P_{\leq N} \phi_{<N} \|_{C^1} \| P_N \psi \|_{C^1} \\
&\lesssim \| P_{\leq N} \phi_{<N} \|_{C^2} N \| P_N \psi \|_{C^1_{1/N}} \\
&\lesssim \| \phi_{<N} \|_{C^2} N^{-19} \| \nabla^{20} \psi \|_{C^0} \\
&\lesssim_{C_0} A^{10} \eps N^{-19} A^{-18}.
\end{align*}
Finally, from \eqref{algebra-1}, Theorem \ref{lp-theory}(iii), \eqref{sim-3} one  has
\begin{align*}
\| \phi_N \cdot Y P_{> N} \psi \|_{C^1} &\lesssim
\| \phi_N \|_{C^1} \|Y P_{> N} \psi \|_{C^1}  \\
&\lesssim N \| \phi_N \|_{C^{40}_{1/N}} N^2 \| P_{>N} \psi \|_{C^2_{1/M}} \\
&\lesssim_{C_0} N (N^{-24} \eps + A^{-8} N^{-20-\alpha} \eps) N^2 N^{-20} \| \nabla^{20} \psi \|_{C^0} \\
&\lesssim_{C_0} N (N^{-24} \eps + A^{-8} N^{-20-\alpha} \eps) N^{-18} A^{-19}.
\end{align*}
Inserting all these estimates, we obtain the claim.
\end{proof}

\section{A little bit of quantitative topology}\label{top-sec}

Let $1 \leq k < D$.  Suppose one has a family $v_1,\dots,v_k \colon H \to \R^D$ of continuous maps such that for each point $p$, $v_1(p),\dots,v_k(p)$ form an orthonormal system in $\R^D$.  Is it always possible to find an additional continuous map $v_{k+1} \colon H \to \R^D$ such that $v_1(p),\dots,v_{k+1}(p)$ is also orthonormal?   As we shall see in the next section, a (quantitative version of) this lifting property will be useful to construct a solution to the low-frequency equation \eqref{lfe}.  

An equivalent way to phrase this question is as follows.  Define the \emph{Steifel manifold} $V_{k,D} \subset \R^{kD}$ to be the space of $k$-tuples $(v_1,\dots,v_k)$ of orthonormal vectors in $\R^D$; this is a smooth compact submanifold of $\R^{kD}$, and the projection map $\pi \colon V_{k+1,D} \to V_{k,D}$ defined by $\pi(v_1,\dots,v_{k+1}) \coloneqq (v_1,\dots,v_k)$ gives $V_{k+1,D}$ the structure of an $S^{D-k-1}$-bundle over $V_{k,D}$.  The question is then whether every continuous map $f \colon H \to V_{k,D}$ has a continuous lift $\tilde f \colon H \to V_{k+1,D}$ (that is, $\tilde f$ is continuous with $\pi \circ \tilde f = f$).

Another equivalent formulation is the following.  Let $E \subset H \times S^{D-1}$ be the set of pairs $(p, v)$ where $p \in H$ and $v$ is a unit vector orthogonal to $v_1(p),\dots,v_k(p)$, and let $\pi' \colon E \to H$ be the projection map $\pi'(p,v) \coloneqq p$.  It is easy to check that $E$ is a fibre bundle over $H$ whose fibres are all homeomorphic to $S^{D-k-1}$.  The question is then whether this fibre bundle has a global section $p \mapsto (p, v_{k+1}(p))$.

For some special values of $(k,D)$, there exist global sections from $V_{k,D}$ to $V_{k+1,D}$, and one can obtain a lift simply by composing the original map $f$ with this section.  For instance, when $(k,D)=(2,3)$, one can simply take $v_3(p)$ to be the cross product of $v_1(p)$ and $v_2(p)$.  Unfortunately, such global sections are very rare: a result of Whitehead \cite{whitehead} shows that these exist\footnote{We thank David Speyer for this reference, which was provided at {\tt mathoverflow.net/questions/314613}.} if and only if $(k,D)$ is equal to $(1,2m)$, $(m-1,m)$, $(2,7)$, or $(3,8)$ for some natural number $m$.  Nevertheless, because the domain $H$ is so low dimensional, and because many of the low-dimensional homotopy groups of the fibres $S^{D-k-1}$ vanish, one can use some very basic obstruction theory to solve the lifting problem when $D-k$ is large:

\begin{proposition}[Non-uniform lifting]\label{nonun}  Suppose $1 \leq k \leq D-4$.  Then every continuous map $f\colon H \to V_{k,D}$ can be lifted continuously to a map $\tilde f\colon H \to V_{k+1,D}$.  
\end{proposition}

\begin{proof}  Using the bundle formulation (with $f = (v_1,\dots,v_k)$), it suffices to construct a global section of $E$ on $H$.
The three-dimensional manifold $H$ has the structure of a CW-complex, and in particular one has a nested sequence $\Delta^0 \subset \Delta^1 \subset \Delta^2 \subset \Delta^3 = H$ of $n$-skeletons $\Delta^n$ of $H$, consisting of the unions of cells of dimension at most $n$.  As $\Delta^0$ is discrete, one can clearly construct a section of $E$ on $\Delta^0$.  It then suffices to show that for each $n=0,1,2$, a section $p \mapsto (p, v_{k+1}(p))$ of $E$ on $\Delta^n$ can be continuously extended to a section of $E$ on $\Delta^{n+1}$.  As continuity is a local property, it suffices to show that any continuous section on the boundary $\partial C^{n+1}$ of an open $n+1$-dimensional cell $C^{n+1}$ in the CW-complex can be continuously extended to the closed cell $\overline{C^{n+1}}$.

Pick a point $p$ in $C^{n+1}$, and let $B$ be a small open ball centred at $p$ in $C^{n+1}$.  The boundary $\partial C^{n+1}$ can be contracted to the boundary of the ball $B$, so by the homotopy lifting property one can extend the section on $\partial C^{n+1}$ to the region $\overline{C^{n+1}} \backslash B$.  On the other hand, if $B$ is small enough, the portion of the bundle $E$ over $B$ trivialises and is thus homeomorphic to $B \times S^{D-k-1}$.  Using this trivialisation, the section on $\partial B$ can be now identified with a continuous map from the $n$-dimensional sphere $\partial B$ to the fibre $S^{D-k-1}$.  Since $D-k-1 \geq 3 > n$, the homotopy group $\pi_n(S^{D-k-1})$ is trivial, and hence this continuous map can be extended continuously to $B$.  Gluing together all these extensions, we obtain a continuous extension of the section to $\overline{C^{n+1}}$ as desired.
\end{proof}

\begin{remark}  Because $H$ is topologically trivial, one could also obtain this lift (without the requirement $k \leq D-4$) by working on a ball $B(0,R)$ and continuously extending $R$ from zero to infinity, using the Gram-Schmidt process as one goes along to keep everything orthogonal; see for instance\footnote{We thank Stan Palasek for these references, which provide yet another link between the arguments here and those used for the isometric embedding problem.} \cite[Section 2.4]{delellis} or \cite[p. 387-388]{nashk}.  However, this argument does not seem to easily extend to the quantitative version that we need below, due to the non-compact nature of $H$ (or equivalently, the unbounded nature of $R$).
\end{remark}

In our application, Proposition \ref{nonun} is not sufficient because we will need some uniform control on the lift $\tilde f$ (in the spirit of Gromov \cite{gromov}).  Fortunately, due to the uniformly bounded geometry of $H$, such uniform control is indeed available:

\begin{proposition}[Uniform lifting]\label{equi}  Suppose $1 \leq k \leq D-4$.  Let ${\mathcal F}$ be a uniformly equicontinuous family of continuous maps $f \colon H \to V_{k,D}$.  Then there is a uniformly equicontinuous family $\tilde {\mathcal F}$ of continuous maps $\tilde f \colon H \to V_{k+1,D}$, such that every $f \in {\mathcal F}$ has a lift $\tilde f \in \tilde {\mathcal F}$.  
\end{proposition}

\begin{proof}  We repeat the proof of Proposition \ref{nonun}, but taking care to obtain uniformly equicontinuous control on all the objects used in the argument.  The main difficulty arises from the non-compact nature of $H$, so we will make our constructions equivariant with respect to the right-action of the cocompact lattice $\Gamma$.

It will be convenient to use a CW-complex of $H$ in which the cells take the form $C \gamma$ with $\gamma \in \Gamma$ and $C$ drawn from a finite list of polytopes in $\R^3 \equiv H$.  The precise choice of complex is not important, but one can for instance take the $3$-cells to be ``cubes'' of the form
$$\{ [x,y,z] \in H: x,y,z \in (0,1)\} \gamma,$$
the $2$-cells to be either ``squares'' of the form
$$  \{ [x,y,0] \in H: x,y \in (0,1) \} \gamma, \{ [0,y,z] \in H: y,z \in (0,1) \} \gamma$$
or ``triangles'' of the form
$$ \{ [x,0,z] \in H: x \in (0,1), 0 < z < 1-x \} \gamma, \{ [x,0,z] \in H: x \in (0,1), 1-x < z < 1 \} \gamma,$$
the $1$-cells to be ``line segments'' of the form
\begin{align*}
&\{ [x,0,0] \in H: x \in (0,1) \} \gamma \\
&\{ [0,y,0] \in H: y \in (0,1) \} \gamma \\
&\{ [0,0,z] \in H: z \in (0,1) \} \gamma \\
&\{ [x,0,-x] \in H: x \in (0,1) \} \gamma
\end{align*}
and the $0$-cells to be the individual points in $\Gamma$. As before, we define the $n$-skeleta $\Delta^n$ for $n=0,1,2,3$ as the union of all cells of dimension at most $n$.

Let $(v_1,\dots,v_k)$ be an element of ${\mathcal F}$, and let $E$ be the bundle constructed previously.  Our task is to construct a global section $p \mapsto (p, v_{k+1}(p))$ of this bundle that lies in a uniformly equicontinuous family as $(v_1,\dots,v_k)$ ranges over ${\mathcal F}$.  On the $0$-skeleton $\Delta^0 = \Gamma$, this is easily achieved by selecting $(p, v_{k+1}(p))$ arbitrarily from the fibre of $E$ at $p$ for each $p \in \Delta^0$.  It thus suffices to show for each $n=0,1,2$ that any section of $E$ on $\Delta^n$ that lies in a uniformly equicontinuous family can be extended to $\Delta^{n+1}$, with the extension also lying in a uniformly equicontinuous family.

As before, it suffices to work on each cell $C^{n+1} \gamma$, that is to say for each $C^{n+1}$ in the above list of $n+1$-polytopes and each $\gamma \in \Gamma$, every section $p \mapsto (p, v_{k+1}(p))$ of $E$ on $\partial C^{n+1} \gamma$ lying in a uniformly equicontinuous family can be extended to $\overline{C^{n+1}} \gamma$ while still lying in a uniformly equicontinuous family; it is easy to see that by gluing these extensions for all cells $C^{n+1} \gamma$ we obtain an extension to $\Delta^{n+1}$ that still lies in a uniformly equicontinuous family.

As the metric on $H$ is right-invariant, arbitrary translations of functions in a uniformly equicontinuous family still form a uniformly equicontinuous family, so we may normalise $\gamma=0$, thus $C^{n+1}$ is now a polytope drawn from a finite list.  As before, we pick a point $p$
in the interior of $C^{n+1}$ (e.g., the centroid), and let $B$ be a small ball centred at $p$.  As $(v_1,\dots,v_k)$ belongs to a uniformly equicontinuous family, we can choose $B$ uniformly over this family so that the bundle $E$ over $B$ can be trivialised to $B \times S^{D-k-1}$, and furthermore the trivialisation map is also uniformly equicontinuous.  

We can extend the section $p \mapsto (p, v_{k+1}(p))$ on $\partial C^{n+1}$ to the region $C^{n+1} \backslash B$ by taking an arbitrary smooth connection $\nabla$ of the bundle of $V_{k+1,D}$ over $V_{k,D}$, pulling it back to $C^{n+1} \backslash B$, and then following that connection along the inward radial vector field to $p$, which connects each point of the polytope boundary $\partial C^{n+1}$ to a unique point in the sphere $\partial B$.  One can check that this extension lies in a uniformly equicontinuous family.  The remaining task is to extend the section from $\partial B$ to $B$ in a uniformly equicontinuous fashion.  Using the trivialisation, and rescaling $B$ to be the unit ball, the problem then reduces to the following: given a continuous map $f \colon \partial B_{\R^{n+1}}(0,1) \to S^{D-k-1}$ in a uniformly equicontinuous family, construct an extension $\tilde f \colon \overline{B_{\R^{n+1}}(0,1)} \to S^{D-k-1}$ that also lies in a uniformly equicontinuous family.

As the homotopy group $\pi_n(S^{D-k-1})$ is trivial, every $f \colon \partial B_{\R^{n+1}}(0,1) \to S^{D-k-1}$ in the family has at least one continuous extension $\tilde f \colon \overline{B_{\R^{n+1}}(0,1)} \to S^{D-k-1}$; the issue is that of uniform equicontinuity of $\tilde f$.  But by the Arzel\'a-Ascoli theorem, the family of $f$ is precompact in the uniform topology.  Thus it suffices to show that for each continuous $f_0 \colon \partial B_{\R^n}(0,1) \to S^{D-k-1}$, all continuous $f \colon \partial B_{\R^{n+1}}(0,1) \to S^{D-k-1}$ in a sufficiently small neighbourhood of $f_0$ in the uniform topology and in a uniformly equicontinuous family, have a continuous extension $\tilde f \colon \overline{B_{\R^{n+1}}(0,1)} \to S^{D-k-1}$ that also lies in a uniformly equicontinuous family, where this latter family is permitted to depend on $f_0$.  But one can achieve\footnote{In lieu of this compactness argument, one can also use the literature on quantitative null-homotopy \cite{gromov}, \cite{chambers}, \cite{cdw}, which would give a more explicit dependence on constants.} this by letting $\tilde f_0 \colon \overline{B_{\R^{n+1}}(0,1)} \to S^{D-k-1}$ be an arbitrary continuous extension of $f_0$ and then defining $\tilde f \colon \overline{B_{\R^{n+1}}(0,1)} \to S^{D-k-1}$ in polar coordinates by the formula
$$ \tilde f(r\omega) \coloneqq \pi( \tilde f_0(r\omega) + \eta(r) (f(\omega) - \tilde f_0(r\omega) ) )$$
for $0 \leq r \leq 1$ and $\omega \in \partial B_{\R^n}(0,1)$, where $\eta \colon \R \to [0,1]$ is a continuous function (depending on $\tilde f_0$) supported on a sufficiently small neighbourhood of $1$ with $\eta(1)=1$, and $\pi \colon \R^{D-k} \backslash \{0\} \to S^{D-k-1}$ is the radial projection to the unit sphere.   One easily checks that for $f$ close enough to $f_0$ in the uniform topology, $\tilde f$ is well-defined (with the argument of $\pi$ avoiding the origin) and is a continuous extension of $f$ that lies in a uniformly equicontinuous family, giving the claim.
\end{proof}

Now we establish a variant using the $C^j$ norms:

\begin{corollary}[$C^j$ lifting]\label{Steifel}  Let $1 \leq k \leq D - 4$ and let $j \geq 1$.  Let $v_1,\dots,v_k \colon H \to S^{D-1}$ be smooth functions with 
$$ \| \nabla v_i \|_{C^j} \leq K$$
for some $0 < K < 1$ and all $i=1,\dots,k$, such that for every $p$, $v_1(p),\dots,v_k(p)$ form an orthonormal system in $\R^D$.  Then there is a smooth function $v_{k+1} \colon H \to S^{D-1}$ with
$$ \| \nabla v_{k+1} \|_{C^j} \lesssim_{D,j} K $$
such that for every $p \in H$, $v_{k+1}(p)$ is orthogonal to $v_1(p),\dots,v_k(p)$. 
\end{corollary}

\begin{proof}  We may rescale so that each $v_i$, $i=1,\dots,k$ has a Lipschitz constant of $O(1)$, and $\| \nabla v_i \|_{C^j_K} \leq 1$.  In particular, for fixed $K$, $(v_1,\dots,v_k)$ lies in a uniformly equicontinuous family independent of $K \in (0,1)$.  Applying Proposition \ref{equi}, we can find a continuous map $v_{k+1} \colon H \to S^{D-1}$ in a uniformly equicontinuous family such that $v_{k+1}(p)$ is orthogonal to $v_1(p),\dots,v_k(p)$ for all $p \in H$.

The remaining task is to ``smooth out'' $v_{k+1}$ to obtain the modification $v'_{k+1} \colon H \to S^{D-1}$ that obeys the required property $\| \nabla v'_{k+1} \|_{C^j_K} \lesssim_{D,j} 1$. Let $\sigma > 0$ be a small constant depending only on $D$ to be chosen later, and let $r>0$ be sufficiently small depending on $\sigma,D$.  By uniform equicontinuity, we see that $v_i(p') = v_i(p) + O(\sigma)$ whenever $d(p,p') \leq r$ and $i=1,\dots,k+1$.  By applying a smooth partition of unity, we can write $1 = \sum_{\gamma \in \Gamma} \varphi(p \gamma^{-1})$ for all $p \in H$ and a smooth compactly supported function $\varphi$; dilating this by $r$, we see that $1 = \sum_{\gamma \in \delta_r \Gamma} \varphi \circ \delta_r^{-1}( p \gamma^{-1} )$.  

For $\gamma \in \delta_r \Gamma$ and $p \in H$ with $d(p,\gamma) = O(r)$, we see that 
$$v_i(p) \cdot v_{k+1}(\gamma) = v_i(p) \cdot \tilde v_{k+1}(p) + O(\sigma) = O(\sigma)$$
for $i=1,\dots,k$.  Thus, if we define
\begin{equation}\label{orange}
 \tilde v_{k+1}(p) \coloneqq \sum_{\gamma \in \delta_r \Gamma} \varphi \circ \delta_r^{-1}(p \gamma^{-1}) v_{k+1}(\gamma)
\end{equation}
then
$$v_i(p) \cdot \tilde v_{k+1}(p) = O(\sigma);$$
also, we have the derivative bounds
$$ \| \nabla \tilde v_{k+1} \|_{C^j_K} \lesssim_{j,r} \eps$$
thanks to many applications of the chain rule; since all the unit vectors $v_{k+1}(\gamma)$ that give a non-zero contribution to \eqref{orange} lie within $O(\sigma)$ of each other, we have 
$$ |\tilde v_{k+1}(p)| = 1 + O(\sigma).$$
Thus, if we apply the Gram-Schmidt process to define
$$ v'_{k+1}(p) := \frac{1}{|w_{k+1}(p)|} w_{k+1}(p)$$
where 
$$w_{k+1}(p) := \tilde v_{k+1}(p) - \sum_{i=1}^k (\tilde v_{k+1}(p) \cdot v_i(p)) v_i(p)$$
then we see that $|w_{k+1,p}| = 1 + O_D(\sigma)$, so if $\sigma$ is small enough, $v'_{k+1}(p)$ is a well-defined unit vector orthogonal to $v_1(p), \dots, v_k(p)$, and from the chain rule and product rule we obtain the bounds
$$ \| \nabla v'_{k+1} \|_{C^j_K} \lesssim_{j,r,D} 1 $$
giving the claim.  (Here the hypothesis $K<1$ is needed to ensure that the modified norm $\|v\|_{C^0} + \|\nabla v \|_{C^j_K}$ obeys the algebra property \eqref{algebra-1}; this norm can be used for instnce to control $|w_{k+1}|$ and its reciprocal.)
\end{proof}

\section{Conclusion of the argument}\label{conclude}

Now that Proposition \ref{perturb} and Corollary \ref{Steifel} are established, we can return to the proof of Proposition \ref{induct}.  We will explicitly construct a function $\tilde \phi \colon  H \to \R^{29}$ that solves a low-frequency equation \eqref{lfe} and obeys most of the properties of Proposition \ref{induct}; the final solution $\phi \colon  H \to \R^{29}$ required by Proposition \ref{induct} will then be obtained by applying Proposition \ref{perturb} with $F=0$. 

Let the notation and hypotheses be as in Proposition \ref{induct}. The function $\tilde \phi$ will take the form
\begin{equation}\label{tpho}
 \tilde \phi(p) = U(p)( \phi^0(p) )
\end{equation}
for $p \in H$, where $\phi^0 \colon H \to \R^{20}$ is the function from Proposition \ref{prop}, and $U(p) \colon \R^{20} \to \R^{29}$ is a linear isometry varying smoothly (and slowly) in $p$ in a manner dependent on $P_{\leq N_0} \psi$, and is in particular chosen to make the bilinear form $B( \tilde \phi, P_{\leq N_0} \psi)$ vanish.  Thanks to Corollary \ref{Steifel}, we can construct $U(p)$ in a straightforward fashion:

\begin{lemma}[Construction of $U$]\label{ucon}  For each $p$, there exists a linear isometry $U(p) \colon \R^{20} \to \R^{29}$ such that
\begin{equation}\label{uxp}
 (U(p) s) \cdot W P_{\leq N_0} \psi(p) = 0
\end{equation}
for all $s \in \R^{20}$ and $W \in \{ X, Y, Z, XX, YY, XY \}$ (and hence also $W = YX$, thanks to \eqref{heisenberg}).  Furthermore, $U(p)$ depends smoothly on $p$ with
$$ \| \nabla U \|_{C^{39}} \lesssim_{N_0} \frac{1}{A}.$$
\end{lemma}

\begin{proof}  Let $W_1,W_2,W_3,W_4,W_5,W_6$ denote the rescaled differential operators 
$$M^{-1} X,M^{-1} Y,A Z,A XX,A YY,A XY$$
respectively.  For each $p$, let $(v_1(p), \dots, v_6(p)) \in V_{6,29} $ be the orthonormal system formed by applying the Gram-Schmidt process to the vectors 
$w_i(p) := W_i P_{\leq N_0} \psi(p)$ for $i=1,\dots,6$, thus (omitting dependence on $p$ for brevity)
$$ v_i \coloneqq \frac{|\bigwedge_{j<i} w_j|}{|\bigwedge_{j \leq i} w_j|} \left( w_i - \sum_{j<i} (w_i \cdot v_j) v_j \right)$$
for $i=1,\dots,6$.  From \eqref{j-hold}, \eqref{xyp-comp}, Theorem \ref{lp-theory}(iii) one has
$$ w_i = W_i \psi + O_{N_0}( A^{-1} M^{-1} ) = O_{N_0}(1)$$
for $i=1,2$, and
$$ w_i = W_i \psi + O_{N_0}( A^{-1} ) = O_{N_0}(1)$$
for $i=3,4,5,6$.  Thus, we have $w_i = O_{N_0}(1)$ for all $1 \leq i \leq 6$.
From \eqref{free} and the triangle inequality we have
$$\left|\bigwedge_{j \leq 6} w_j\right| \gtrsim_{N_0} 1$$
and then by Cauchy-Schwarz we also have
$$
\left|\bigwedge_{j \leq i} w_j\right| \sim_{N_0} 1
$$
for all $0 \leq i \leq 6$ (cf. the proof of Lemma \ref{tpsi}).  Also, from \eqref{j-hold}, Theorem \ref{lp-theory}(iii) one has the bounds
$$ \|w_i \|_{C^0} + A \| \nabla w_i \|_{C^{39}} \lesssim_{N_0} 1$$
for $1 \leq i \leq 6$.  In particular, from \eqref{algebra-1} one has
$$ \left\|\bigwedge_{j \leq i} w_j\right\|_{C^0} + A \left\|\nabla \bigwedge_{j \leq i} w_j\right\|_{C^{39}}  \sim_{N_0} 1.$$

From these bounds and the quotient and product rules, we see from an induction on $i$ that
$$
\|v_i\|_{C^0} + A \|\nabla v_i \|_{C^{39}} \lesssim_{N_0} 1$$
for $i=1,\dots,6$.  Applying Corollary \ref{Steifel} $20$ times, noting that $6+20 = 29-3$, we may then find smooth maps $v_7,\dots,v_{26} \colon H \to \R^{29}$ such that
$$
\|v_i\|_{C^0} + A \|\nabla v_i \|_{C^{39}} \lesssim_{N_0} 1
$$
for $i=1,\dots,26$, and such that $v_1(p),\dots,v_{26}(p)$ are orthonormal for all $p \in H$.  If we then define $U(p)$ to be the map
$$
U(p)(s_1,\dots,s_{20}) \coloneqq \sum_{i=1}^{20} s_i v_{i+6}(p)$$
then the claim follows.
\end{proof}

Now define $\tilde \phi \colon  H \to \R^{29}$ by the formula \eqref{tpho}.  From Lemma \ref{ucon}, \eqref{tpho}, \eqref{algebra-1}, and Proposition \ref{prop} we have
\begin{equation}\label{tpc}
 \| \tilde \phi \|_{C^{40}} \lesssim 1.
\end{equation}
Next, we compute $B( \tilde \phi, P_{\leq N_0} \psi )$.  The first component $X \tilde \phi \cdot X P_{\leq N_0} \phi$ expands using the product rule as
$$ X (\tilde \phi \cdot X P_{\leq N_0} \phi) - \tilde \phi \cdot XX P_{\leq N_0} \phi.$$
But both terms vanish thanks to \eqref{uxp}, \eqref{tpho}.  Similarly for the second component of $B(\tilde \phi, P_{\leq N_0} \psi)$, and so we have the low frequency equation \eqref{lfe}.  We may now apply Proposition \ref{perturb} to locate a smooth solution $\phi \colon  H \to \R^{29}$ to the equation \eqref{x-ortho} with
\begin{equation}\label{ax-new}
 \| \phi - \tilde \phi \|_{C^{20,\alpha}} \lesssim_{C_0} A^{-10}
\end{equation}
and 
\begin{equation}\label{ax-2-new}
 \| X(\phi-\tilde \phi) \cdot Y \psi \|_{C^0}, \| Y(\phi-\tilde \phi) \cdot X \psi \|_{C^0} \lesssim_{C_0} A^{-10}.
\end{equation}
To finish the proof of Theorem \ref{induct}, we need to verify the conclusions \eqref{xyp}-\eqref{x-ortho} of that theorem.  The claim \eqref{x-ortho} was obtained by construction, and the claim \eqref{j-hold-phi} is immediate from \eqref{ax-new}, \eqref{tpc}.  Now we turn to \eqref{xyp}.  For any $p \in H$, we see from \eqref{ax-new} that
$$ X \phi(p) = X \tilde \phi(p) + O_{N_0}(A^{-10}).$$
From \eqref{tpho}, the product rule, Proposition \ref{prop}, and Lemma \ref{ucon} we have
$$X \tilde \phi(p) = U(p)( X \phi_0(p)) + O_{N_0}(A^{-1}).$$
From Proposition \ref{prop} we have $|X \phi_0(p)| \sim 1$.  Since $U(p)$ is an isometry, we conclude \eqref{xyp} for $X \phi(p)$, and a similar argument gives \eqref{xyp} for $Y \phi(p)$ also.  For future reference, we observe that this argument and Proposition \ref{prop} also gives the bound
\begin{equation}\label{xyphi}
 |X \phi(p) \wedge Y \phi(p)| \sim 1.
\end{equation}

Now we establish the delicate estimate \eqref{more-ortho}. Fix $p \in H$; for brevity we omit the explicit dependence on $p$.  We begin\footnote{We thank the anonymous referee for this simplified version of the argument.} with an estimation of the inner product $X \phi \cdot Y \psi$ (and also $X \psi \cdot Y \phi$).  
The expression
$$X \tilde \phi \cdot Y P_{\leq N_0} \psi = X ( U \phi_0 \cdot Y P_{\leq N_0} \psi ) - U \phi_0 \cdot XY P_{\leq N_0}\psi$$
vanishes by Proposition \ref{ucon}, and hence
$$ X \phi \cdot Y \psi = X(\phi -\tilde \phi) \cdot Y \psi + X \tilde \phi \cdot Y P_{>N_0} \psi.$$
By \eqref{ax-2-new}, the first term on the right-hand side is $O_{N_0}( A^{-10} )$.  From Theorem \ref{lp-theory}(iii), \eqref{j-hold} one has
\begin{equation}\label{yo}
|Y P_{>N_0} \psi| \lesssim \| \nabla^2 \psi \|_{C^0} \lesssim_{C_0} A^{-1}
\end{equation}
and hence by \eqref{tpc}
$$
 |X \tilde \phi \cdot Y P_{>N_0} \psi| \lesssim_{C_0} A^{-1}.
$$
From this (and analogous arguments for $X \psi \cdot Y \phi$ we conclude that
\begin{equation}\label{xy1}
 X \phi \cdot Y \psi, X \psi \cdot Y \phi = O_{C_0}\left(\frac{1}{A} \right).
\end{equation}
Also from \eqref{x-ortho} we also have
\begin{equation}\label{xy2}
 X \phi \cdot X \psi = Y \psi \cdot Y \phi = 0.
\end{equation}  
Meanwhile, from \eqref{xyp-comp}, \eqref{j-hold-phi} we have
\begin{equation}\label{xy3}
 X \psi, Y \psi = O_{C_0}(M); \quad X \phi, Y \phi = O(1).
\end{equation}

We split 
$$X(\psi+\phi) \wedge Y(\psi+\phi) = X \psi \wedge Y \psi + (X \psi \wedge Y \phi + X \phi \wedge Y \psi + X \phi \wedge Y \phi)$$
and hence by the cosine rule
\begin{align*}
 |X(\psi+\phi) \wedge Y(\psi+\phi)|^2 - |X \psi \wedge Y \psi|^2 &= |X \psi \wedge Y \phi + X \phi \wedge Y \psi + X \phi \wedge Y \phi|^2 \\
&\quad - 2 \langle X \psi \wedge Y \psi, X \psi \wedge Y \phi + X \phi \wedge Y \psi + X \phi \wedge Y \phi\rangle.
\end{align*}

From \eqref{xy1}, \eqref{xy2}, \eqref{xy3}, and the depolarised Cauchy-Binet formula \eqref{depolar} we have
$$\langle X \psi \wedge Y \psi, X \psi \wedge Y \phi + X \phi \wedge Y \psi + X \phi \wedge Y \phi\rangle = O_{C_0} \left( \frac{M^2}{A} \right)$$
and so it will suffice to show that
$$ |X \psi \wedge Y \phi + X \phi \wedge Y \psi + X \phi \wedge Y \phi| \gtrsim C_0^{-2} M.$$
Taking wedge products with $Y\phi$ and using Cauchy-Schwarz and \eqref{xy3}, it suffices to show that
$$ |X \phi \wedge Y \phi \wedge Y \psi|^2 \gtrsim C_0^{-4} M^2.$$
From two applications of the Cauchy-Binet formula \eqref{depolar} together with \eqref{xy1}, \eqref{xy2}, \eqref{xy3} we see that
$$ |X \phi \wedge Y \phi \wedge Y \psi|^2  = |X \phi \wedge Y \phi|^2 |Y \psi|^2 + O_{C_0} \left( \frac{1}{A} \right ).$$
The claim now follows from \eqref{xyp-comp}, \eqref{xyphi}.

Finally, we verify \eqref{free-inc}.  From \eqref{j-hold}, \eqref{ax-new}, \eqref{tpho}, and Lemma \ref{ucon} we have (omitting dependence on $p$ for brevity)
$$
 W(\psi + \phi) = O_{C_0}(A^{-1}) + W \tilde \phi + O(A^{-10}) = U( W \phi_0 ) + O_{N_0}(A^{-1})$$
for all $W \in \{Z, XX, YY, XY \}$.
From Proposition \ref{prop}, $W \phi_0$ has norm $\sim 1$.  Thus
$$
\bigwedge_{W = Z, XX, YY, XY} W (\psi+\phi) = \omega + O_{N_0}(A^{-1})$$
where
$$ \omega \coloneqq \bigwedge_{W = Z, XX, YY, XY} U(W \phi_0) $$
and so (since $X(\psi+\phi), Y(\psi+\phi) = O_{C_0}(M)$) it will suffice to establish the bound
$$
 |X(\psi+\phi) \wedge Y(\psi+\phi) \wedge \omega| \gtrsim C_0^{-12} M^2.$$
Note from Proposition \ref{prop} that $\omega = O(1)$.  By Cauchy-Schwarz and \eqref{xyp-comp} it will thus suffice to show that
\begin{equation}\label{left}
\langle X P_{\leq N_0} \psi \wedge Y P_{\leq N_0} \psi \wedge \omega, X(\psi+\phi) \wedge Y(\psi+\phi) \wedge \omega \rangle \gtrsim C_0^{-8} M^4.\end{equation}
By Lemma \ref{ucon}, all the vectors $U(W \phi_0)$ comprising $\omega$ are orthogonal to both $X P_{\leq N_0} \psi$ and $Y P_{\leq N_0} \psi$.  Using the Cauchy-Binet formula \eqref{depolar} twice, the left-hand side can then be written as
$$
\langle X P_{\leq N_0} \psi \wedge Y P_{\leq N_0} \psi, X(\psi+\phi) \wedge Y(\psi+\phi) \rangle |\omega|^2.
$$
As $U$ is an isometry, we see from Proposition \ref{prop} that $|\omega| \gtrsim 1$.  Meanwhile, from \eqref{xy1}, \eqref{xy2}, \eqref{xyp-comp}, \eqref{yo} (and the analogue for $X P_{>N_0} \psi$) we have
$$ W P_{\leq N_0} \psi  \cdot W'(\psi+\phi) = W \psi \cdot W' \psi + O_{C_0}\left( \frac{M}{A} \right)$$
for $W,W' \in \{X,Y\}$, hence by \eqref{depolar} (and \eqref{xyp-comp}) again
$$
\langle X P_{\leq N_0} \psi \wedge Y P_{\leq N_0} \psi, X(\psi+\phi) \wedge Y(\psi+\phi) \rangle = |X \psi \wedge Y \psi|^2 +
O_{C_0}\left( \frac{M^3}{A} \right).
$$
The claim now follows from \eqref{xyp-nondeg}.  This (finally!) concludes the proof of Proposition \ref{induct} and thus Theorem \ref{th1}.



\projects{\noindent The author was supported by a Simons Investigator grant, the James and Carol Collins Chair, the Mathematical Analysis \&
Application Research Fund Endowment, and by NSF grant DMS-1266164.}

\end{document}